\documentclass[11pt]{amsart}

\usepackage[T1]{fontenc}
\usepackage{amsmath}
\usepackage{amsfonts}
\usepackage{amssymb}
\usepackage{amsopn}
\usepackage{amscd}
\usepackage{amsthm}
\usepackage[latin1]{inputenc}
\usepackage[all]{xy}
\usepackage{graphicx}

\newtheorem{thm}{Theorem}[section]
\newtheorem{prop}[thm]{Proposition}
\newtheorem{lem}[thm]{Lemma}

\theoremstyle{definition}
\newtheorem{definition}[thm]{Definition}
\newtheorem{example}[thm]{Example}

\theoremstyle{remark}

\newtheorem{remark}[thm]{Remark}

\newtheorem{remarks}[thm]{Remarks}

\numberwithin{equation}{section}

\newcommand{\C}{\mathbb{C}}
\newcommand{\R}{\mathbb{R}}
\DeclareMathOperator{\dist}{dist}

\DeclareMathOperator{\diam}{diam}
\newcommand{\dd}{\mathrm{d}}
\newcommand{\card}{\mathrm{Card}}
\DeclareMathOperator{\supp}{supp}
\newcommand{\N}{\mathbb N}
\newcommand{\Z}{\mathbb Z}
\newcommand{\Q}{\mathbb Q}

\begin{document}

\title{Distribution of points and Hardy type inequalities in spaces of homogeneous type}

\author{E. Routin}
\address{Eddy Routin, Université Paris-Sud, laboratoire de Math\'ematiques, UMR 8628, Orsay F-91405; CNRS, Orsay, F-91405; France}
\email{eddy.routin@normalesup.org}

\begin{abstract}
In the setting of spaces of homogeneous type, we study some Hardy type inequalities, which notably appeared in the proofs of local $T(b)$ theorems as in \cite{AR}. We give some sufficient conditions ensuring their validity, related to the geometry and distribution of points in the homogeneous space. We study the relationships between these conditions and give some examples and counterexamples in the complex plane.
\end{abstract}



\maketitle

\tableofcontents

\bigskip

\section{Introduction}

\bigskip
The goal of this paper is to study, in the setting of a space of homogeneous type, what we call Hardy type inequalities. They notably appear in the proofs of local $T(b)$ theorems as in \cite{Hofmann}, \cite{AR}, where they play a crucial role to estimate some of the matrix coefficients involved in the arguments. The prototype of a Hardy type inequality is the following in the Euclidean space of dimension $1$:
\begin{equation*}\label{hardydim1}
 \left |  \int_I \int _J {\frac{f(y)g(x)}{x-y}\dd y \dd x}  \right | \leq C \left ( \int_I{|f(y)|^{\nu} \dd y} \right )^ {\frac{1}{\nu}} \left ( \int_J{|g(x)|^{\nu'} \dd x} \right )^ {\frac{1}{\nu'}},
 \end{equation*}
where $I,J$ are adjacent intervals, $\supp f \subset I$, $\supp g \subset J$, and $1<\nu<\infty$, $\nu' = \frac{\nu}{\nu-1}$. The integral here is absolutely convergent, and this is an immediate consequence of the boundedness of the Hardy operator $H(f)(x) = \frac{1}{x} \int_0^x {f(t)\dd t}$ (hence our terminology). The above $1-$dimensional inequality easily extends to the Euclidean space in any dimension: for every $1<\nu< +\infty$, there exists $C<+\infty$ such that for every disjoint dyadic cubes $Q,Q'$ in $\R^n$, every function $f \in L^{\nu}(Q)$ supported on $Q$, $g \in L^{\nu'}(Q')$ supported on $Q'$, the following integral is absolutely convergent and we have
\begin{equation*}\label{hardyrn}
\left |  \int_Q \int _{Q'} {\frac{f(y)g(x)}{|x-y|^n}\dd x\dd y}  \right | \leq C \|f\|_{L^{\nu}(Q) } \|g\|_{L^{\nu'}(Q')}.
\end{equation*}
This estimate, well known by the specialists, follows from the $1-$dimensional one by expressing the fact that the singularity in the integral is supported along one direction, transverse to an hyperplane separating the disjoint cubes $Q,Q'$.

A similar result in the setting of a space of homogeneous type is to be expected, that is
\[ \left | \int_I \int_J {\frac{f(y)g(x)}{\mathrm{Vol}(B(x,\dist(x,y)))} d\mu(x) d\mu(y)} \right | \leq C \|f\|_{L^{\nu}(I,d\mu)} \|g\|_{L^{\nu'}(J,d\mu)},  \]
where $I\cap J = \varnothing,$ and $1<\nu<+\infty$, $1/\nu + 1/{\nu'} = 1$. Expectedly, it turns out that it holds without any restriction if $I,J$ are Christ's dyadic cubes (in the sense of \cite{Christ}), even if the previous argument cannot be valid as the dyadic cubes in such a space do not follow any geometry. It seemed not to have been noticed in the literature before our work with P. Auscher \cite{AR}. It relies in particular on the small layers for dyadic cubes. However, if $I$ is a ball $B$ and $J$, say, $2B \backslash B$, then it is not clear in general. It clearly depends on how $B$ and $2B \backslash B$ see each other through their boundary.  In \cite{AR}, we came up with some small boundary hypothesis on the space of homogeneous type (called the relative layer decay property) ensuring that the inequality was satisfied. We also showed that this property held in all doubling complete Riemannian manifolds, geodesic spaces and more generally in any monotone geodesic space of homogeneous type. The latter notion arose in geometric measure theory from the work of R. Tessera \cite{Tess}, and was recently proved by Lin, Nakai and Yang \cite{LNY} to be equivalent to a chain ball notion introduced by S. Buckley \cite{Buck}.

We continue this study in the present paper, investigating further these different conditions and the relationships they entertain. It appears that they are all connected to the way points are distributed in the homogeneous space. We produce some interesting examples and \mbox{counterexamples} in the complex plane (Theorem \ref{diagram}, see Section $2$). A natural question that also arises is the following: if the Hardy type inequality for balls is satisfied for a fixed couple of exponents, is it satisfied for every couple of exponents? We show the answer is positive if the homogeneous space satisfies some additional hypothesis (Proposition \ref{Hardy1couple}).

The paper is organized as follows. We give some basic definitions, recall the results already obtained in \cite{AR} and present our results in Section $2$. We give the proof of Proposition \ref{Hardy1couple} in Section $3$. We then devote Section $4$ to the layer decay and annular decay properties that appeared in \cite{AR}. We study some geometric properties ensuring that the latter are satisfied in Section $5$. Finally, we present in Sections $6$ and $7$ various examples and counterexamples in the complex plane, inspired from a curve conceived by R. Tessera in \cite{Tess}. 

This work is part of a doctorate dissertation that was conducted at Université Paris-Sud under the supervision of P. Auscher. The author would like to warmly thank him for his kind support. I also thank T. Hytönen for his insightful comments and discussions related to this work.

\section{Definitions and main results}

\bigskip

Throughout this paper, we will work in the setting of a space of homogeneous type, that is a triplet $(X,\rho,\mu)$ where $X$ is a set equipped with a metric $\rho$ and a non-negative Borel measure $\mu$ for which there exists a constant $C_D <+\infty$ such that all the associated balls $B(x,r)=\{y \in X ; \rho(x,y) <r \}$ satisfy the doubling property
\[ 0 < \mu (B(x,2r)) \leq C_D \mu (B(x,r)) < \infty  \]
for every $x \in X$ and $r>0$. We suppose that $\mu(X) \in ]0,+\infty]$, and we allow the presence of atoms in $X$, that is points $x \in X$ such that $\mu(\{x\}) \neq 0$.

\medskip

\begin{remark}
One usually only assumes that $\rho$ is a quasi-distance on $X$ in the definition of a space of homogeneous type in the sense of Coifman and Weiss \cite{CoifWeiss}. For the sake of simplicity, we limit ourselves to the metric setting. However, our work can easily be carried out to the quasi-metric setting, though one then has to assume the balls to be Borel sets to give a sense to the objects we will define in the following. Note that this is not necessarily the case as the quasi-distance, in contrast to a distance, may not be Hölder-regular, and quasi-metric balls might not be open nor even Borel sets with respect to the topology defined by the quasi-distance. Other kind of assumptions and arguments appeared in the literature, see for example \cite{AHyt} for a discussion on the subject.
\end{remark}

\medskip

We will use the notation $A\lesssim B$ (resp. $A \eqsim B$) to denote the estimate $A\leq CB$ (resp. $(1/C) B \leq A \leq C B$) for some absolute constant $C$ which may vary from line to line. Denote by $\supp f$ the support of a function $f$ defined on $X$, $\diam E$ the diameter of a subset $E \subset X$, $\bar{E}$ the topological closure of a set $E \subset X$, $\card \, I$ the cardinal of a finite set $I$, $|E|$ the Lebesgue measure of a set $E \subset \R^n$, and $\rho(E,F) = \inf_{x \in E, y \in F} {\rho(x,y)}$ the distance between two subsets $E,F \subset X$.

For $1 \leq p \leq \infty $, let $p' = \frac{p}{p-1}$ be the dual exponent of $p$. The space of $p$-integrable complex valued functions on $X$ with respect to $\mu$ is denoted by $L^p(X)$, the norm of a function $f \in L^p(X)$ by $\|f\|_p$, the duality bracket given by $\langle f,g\rangle = \int_X{fg d\mu}$ (we do mean the bilinear form), and the mean of a function $f$ on a set $E$ denoted by $[f]_E = \mu(E)^{-1} \int_E{f d\mu}.$

Finally, for any $x,y \in X$, we set
\[ \lambda (x,y) = \mu (B(x, \rho (x,y))). \]
It is easy to see that, because of the doubling property, $\lambda(x,y)$ is comparable to $\lambda(y,x)$, uniformly in $x, y \in X.$

\bigskip

The following result, due to M. Christ (see \cite{Christ}), states the existence of sets analogous to the dyadic cubes of $\R^n$ in a space of homogeneous type.
\begin{lem}\label{cubes}
There exist a collection of open subsets $\{ Q_{\alpha}^j \subset X: j \in \mathbb{Z}, \alpha \in I_j \}$, where $I_j$ denotes some (possibly finite) index set depending on $j$, and constants $0 < \delta < 1$, $a_0 > 0$, $\eta > 0$, and $C_1,C_2 < +\infty$ such that
\begin{enumerate}
\item For all $j \in \mathbb{Z}$, $\mu(\{ X \backslash \bigcup_{\alpha \in
I_j}{Q_{\alpha}^j} \} ) = 0$.
\item If $j < j'$ , then either $Q_{\beta}^{j'} \subset
Q_{\alpha}^{j}$, or $Q_{\beta}^{j'} \cap Q_{\alpha}^{j} = \varnothing$.
\item For each $(j',\beta)$ and each $j<j'$ there is a unique $\alpha$ such that $Q_{\beta}^{j'} \subset Q_{\alpha}^{j}$.
\item For each $(j,\alpha)$, we have  $\mathrm{diam}(Q_{\alpha}^j)
\leq C_1 {\delta}^j$.
\item Each $Q_{\alpha}^j$ contains some ball $B(z_{\alpha}^j,a_0
{\delta}^j)$. We say that $z_{\alpha}^j$ is the center of the cube $Q_{\alpha}^j.$
\item Small boundary condition:
\begin{equation} \label{petitefrontiere}
 \mu \left( \left\{ x \in Q_{\alpha}^j: \rho (x,X \backslash Q_{\alpha}^j) \leq t {\delta}^j \right\} \right) \leq C_2 t^{\eta} \mu(Q_{\alpha}^j) \quad \forall j,\alpha, \quad \forall t>0.
 \end{equation}
\end{enumerate}
\end{lem}

\bigskip

\noindent We will call those open sets dyadic cubes of the space of homogeneous type $X$. For a cube $Q=Q^j_{\alpha}$, $j$ is called the generation of $Q$, and we set $l(Q)= \delta^j$. By $(4)$ and $(5)$, $l(Q)$ is comparable to the diameter of $Q$, and we call it, in analogy with $\R^n$, the length of $Q$.  Whenever $Q^{j+1}_{\alpha} \subset Q^j_{\beta}$, we will say that $ Q^{j+1}_{\alpha}$ is a child of $Q^{j}_{\beta}$, and $Q^{j}_{\beta}$ the parent of $Q^{j+1}_{\alpha}$. It is easy to check that each dyadic cube has a number of children uniformly bounded. A neighbor of $Q$ is any dyadic cube $Q'$ of the same generation with $\rho(Q,Q') < l(Q)$. The notation $\widehat{Q}$ will denote the union of $Q$ and all its neighbors. It is clear that $Q$ and $\widehat{Q}$ have comparable measures. It is also easy to check that a cube $Q$ has a number of neighbors that is uniformly bounded.
%
%
%
%
%

\bigskip

Operating in this dyadic setting is often very effective, but as the construction of these dyadic cubes is quite abstract, they do not follow any geometry and can be "ugly" sets in practice, in spite of their nice properties. Thus, with the motivation to state a local $T(b)$ theorem with hypotheses on balls rather than dyadic cubes in \cite{AR}, we rather looked for a Hardy type inequality valid in the setting of balls. Let us precise what we mean with the following definition.

\bigskip

\begin{definition}{Hardy property.}\\ \label{HardyBalls}
Let $(X,\rho,\mu)$ be a space of homogeneous type. We say that $X$ has the Hardy property $(\mathrm{HP})$ if for every $1<\nu< +\infty$, with dual exponent $\nu'$, there exists $C < +\infty$ such that for every ball $B$ in $X$, with $2B$ denoting the concentric ball with double radius, and all functions $f$ supported on $B$, $f\in L^{\nu}(B)$, $g$ supported on $2B \backslash B$, $g \in L^{\nu'}(2B \backslash B)$, we have
\begin{equation}  \label{3etoile} \tag{H}
\int_{B} \int _{2B\backslash B} {\frac{|f(y)g(x)|}{\lambda(x,y)}\dd \mu(x)\dd \mu(y)}  \leq C \|f\|_{\nu}  \|g\|_{\nu'}.
\end{equation}
\end{definition}

\bigskip

\noindent This is Definition $3.4$ of \cite{AR}. One could have equivalently replaced $2B$ by $cB$ for fixed $c>1$ in \eqref{3etoile}. As stated in the introduction, \eqref{3etoile} is always valid if one replaces $B$ by a dyadic cube $Q$ and $2B \backslash B$ by $\widehat{Q}$. This is Lemma $2.4$ of \cite{AR}. Let us remark that this result was crucial to the estimation of some of the matrix coefficients appearing in the argument to prove the local $T(b)$ theorem central to that paper

Things are actually a bit trickier in the setting of balls, and the Hardy property $(\mathrm{HP})$ is not always valid, as we will show in Section $7$. The difficulty owes to the fact that balls obviously do not satisfy in general the nice properties satisfied by the dyadic cubes, and particularly the fact that they have small boundaries. We looked for conditions on the way points are distributed in the space of homogeneous type ensuring that the Hardy property would be satisfied. Our main result is the following:

\begin{thm}\label{diagram}
Let $(X, \rho,\mu)$ be a space of homogeneous type. We have the following diagram of implications in $X$:
\[ \xymatrix{
 & & (\mathrm{HB}) \ar@{=>}@<1ex>[d]_{\,}  \ar@{=>}@<1,5ex>[rd]^{\,}   \\
(\mathrm{M}) \, \, \,  \, \not \!\!\!\!\!\!\!  \ar@{=>}@<1,5ex>[r]  & \eqref{LUAD} \ar@{=>}[l] \, \, \, \, \not \!\!\!\!\!\!\!  \ar@{=>}@<1,5ex>[r]  \ar@{=>}[d]  & \eqref{LUOLD}  \ar@{=>}@<1ex>[d]_{\,}   \ar@{=>}[l]  \ar@{=>}[r]  \ar@{=>}@<1ex>[u]|@{/} \ar@{=>}@<-0,5ex>_{\,}[dl]|@{/}  & (\mathrm{HP}) \ar@{=>}@<-3ex>_{\,}[lu]|@{/}  \\
 &  \eqref{AD}   \, \, \,\,\, \,\, \not \!\!\!\!\!\!\!\!\!  \ar@{=>}@<1,25ex>[r] & \eqref{OLD} \ar@{=>}[l]  \ar@{=>}@<1ex>[u]|@{/}  \ar@{=>}@<-1,5ex>_{\,}[ru]|@{/}&
}\]

\end{thm}

\medskip

\noindent A few comments are in order. \begin{enumerate}

\smallskip

\item Theorem \ref{diagram} sums up our study of sufficient conditions for the Hardy property $(\mathrm{HP})$, and of the relationships between these conditions. This work was initiated in \cite{AR}, where some of the implications in the above diagram have already been proved. Every property except $(\mathrm{HP})$ is related to the geometry and the distribution of points in the homogeneous space. 

\smallskip

\item The layer decay \eqref{OLD} and relative layer decay \eqref{LUOLD} properties on one hand, and the stronger annular decay \eqref{AD} and relative annular decay \eqref{LUAD} properties on the other hand, are geometric conditions both metric and measure related. These properties have all been introduced in \cite{AR}, but let us point out that here we slightly modify \eqref{LUOLD} and \eqref{LUAD}, which does not affect the statements already proved in \cite{AR}. They will be recalled in Section $4$.

\smallskip

\item The monotone geodesic property \eqref{monotone} is purely metric and it states the existence of chains of points between two set points. It will be precisely defined and studied in Section $5$, along with what we call the homogeneous balls property $(\mathrm{HB})$.

\smallskip

\item Sections $6$ and $7$ will be devoted to the presentation of various examples and counter-examples which will provide the false implications in Theorem \ref{diagram}, as well as cases of spaces not satisfying $(\mathrm{HP})$.

\smallskip

\item Observe that, conversely, we do not know necessary conditions for the Hardy property $(\mathrm{HP})$. In particular, does $(\mathrm{HP})$ imply that for every ball $B$ of the homogeneous space, $\mu (\overline{B}\backslash B) =0$ ? We think that the answer should be positive, but we have been unable to prove it. Similarly, does $(\mathrm{HP})$ imply \eqref{LUOLD} ? We think this has to be false, but we have not come up with a counterexample yet.
\end{enumerate}
\bigskip

Our last result deals with a natural question regarding these Hardy type inequalities, inspired from the Calder\'on-Zygmund theory. The question is the following: in a general space of homogeneous type, is it possible to deduce the Hardy property $(\mathrm{HP})$ from the Hardy type inequality \eqref{3etoile} for a particular couple $(p,p')$? We have proved that the answer is positive, provided some rather mild additional hypothesis is assumed on the homogeneous space, as shown by the following result.

\medskip

\begin{prop}\label{Hardy1couple}
Let $(X, \rho,\mu)$ be a space of homogeneous type. Assume that for every ball $B \subset X$, $\mu(\bar{B} \backslash B) =0$. Then, if $X$ satisfies the Hardy type inequality \eqref{3etoile} for one particular couple of exponents $(p,p')$, $X$ has the Hardy property $(\mathrm{HP})$.
\end{prop}

\noindent We will give the proof of this result in Section $3$. Remark that the assumption made on the homogeneous space is in particular true if $X$ satisfies the layer decay property \eqref{OLD}, see Section $4$.


\bigskip

\section{Proof of Proposition \ref{Hardy1couple}}

\bigskip

\begin{proof}
Fix a ball $B=B(z_B,r) \subset X$. We assume that \eqref{3etoile} holds for the couple $(p,p')$. By Fubini's theorem, this implies that one can define $T : L^p(B) \rightarrow L^p(2B \backslash B)$ by the absolutely convergent integral
\[  (Tf) (y) = \int_B \frac{f(x)}{\lambda(x,y)} \dd \mu(x)  \quad \mathrm{for} \, \mathrm{almost} \, \mathrm{every} \ \  y \in 2B \backslash B. \]
We will proceed in three steps.
\bigskip

\noindent $(1)$ The first step in the argument consists in regularizing the kernel $\lambda(x,y)^{-1}$: we show that we can freely assume it satisfies a Lipschitz regularity estimate\footnote{We thank T. Hytönen for this idea which nicely improved our earlier result.}. To do this, let $\varphi$ be a function of the real variable such that $\varphi \in C^1([0,+\infty[)$, $\varphi \geq 0$, $\supp \varphi \subset [1,4]$ and $\int_0^{+\infty} \varphi(t) \frac{\dd t}{t}=1$. For $r\geq 0$, set $\lambda(y,r) = \mu(B(y,r))$. Set
\[ \widetilde{\lambda}(x,y) = \int_0^{+\infty} \lambda(y,r) \, \varphi \left (\frac{\rho(x,y)}{r} \right ) \frac{\dd r}{r}. \]
First, observe that $\widetilde{\lambda}(x,y) $ is comparable to $\lambda(x,y)$, uniformly in $x,y \in X$. Indeed, by an easy change of variable, write
\[   \widetilde{\lambda}(x,y) = \int_0^{+\infty} \lambda \left ( y, \frac{\rho(x,y)}{u} \right )  \varphi(u) \frac{\dd u}{u}.    \]
Because of the support conditions and size estimate of $\varphi$, it comes
\[  \lambda \left (y, \frac{1}{4} \rho(x,y) \right ) \leq \widetilde{\lambda}(x,y) \leq \lambda(y, \rho(x,y)) .  \]
By the doubling property, it follows that, uniformly in $x,y \in X$, we have
\begin{equation}\label{comparable}
\widetilde{\lambda}(x,y)  \eqsim \lambda(x,y).
\end{equation}
Then we say that $ \widetilde{\lambda}^{-1}$ satisfies a Lipschitz regularity estimate in the first variable. Indeed, fix $x,x',y \in X$ such that $\rho(x,x') \leq \rho(x,y)/2$ and $\rho(x,y)>0$. Because of the support conditions and regularity of $\varphi$, we have
\begin{align*}
| \widetilde{\lambda} (x,y) - \widetilde{\lambda} (x',y) | & \leq \int_0^{+\infty} \lambda(y,r) \left | \varphi  \left (\frac{\rho(x,y)}{r} \right ) - \varphi \left (\frac{\rho(x',y)}{r} \right ) \right |   \frac{\dd r}{r} \\
& \lesssim \int_{ r \in [\frac{1}{4}\rho(x,y), \rho(x,y)] \cup  [\frac{1}{4}\rho(x',y), \rho(x',y)] } \lambda(y,r) \frac{|\rho(x,y) - \rho(x',y)|}{r} \frac{\dd r}{r} \\
& \lesssim \rho(x,x') \lambda(x,y) \int_{ \frac{1}{4}\rho(x,y) \leq r \leq \rho(x,y) } \frac{\dd r}{r^2} +  \rho(x,x') \lambda(x',y) \int_{ \frac{1}{4}\rho(x',y) \leq r \leq \rho(x',y) } \frac{\dd r}{r^2}\\
& \lesssim \rho(x,x') \left ( \frac{\lambda(x,y)}{\rho(x,y)}  +  \frac{\lambda(x',y)}{\rho(x',y)} \right ). \\
\end{align*}
But since $\rho(x,x') \leq \rho(x,y)/2$, we have $\rho (x,y) \eqsim \rho(x',y)$ and $\lambda(x,y) \eqsim \lambda(x',y)$, uniformly in $x,x',y$. It follows that
\begin{equation}\label{lipschitzreg}
| \widetilde{\lambda} (x,y)^{-1} - \widetilde{\lambda} (x',y)^{-1} |  = \frac {| \widetilde{\lambda} (x,y) - \widetilde{\lambda} (x',y) |}{\widetilde{\lambda}(x,y) \widetilde{\lambda}(x',y)}  \lesssim \frac{\rho(x,x')}{\rho(x,y)} \frac{1}{\widetilde{\lambda}(x,y)},
\end{equation}
because $\widetilde{\lambda}(x,y) \eqsim \widetilde{\lambda}(x',y) \eqsim \lambda(x,y) $.\\
But because of \eqref{comparable}, one can define an operator $\widetilde{T} : L^p(B) \rightarrow L^p(2B \backslash B)$ by the absolutely convergent integral
\[  (\widetilde{T}f) (y) = \int_B \frac{f(x)}{\widetilde{\lambda}(x,y)} \dd \mu(x)  \quad \mathrm{for} \, \mathrm{almost} \, \mathrm{every} \ \  y \in 2B \backslash B, \]
and, for every $1<\nu<+\infty$, the boundedness of $T : L^{\nu}(B) \rightarrow L^{\nu}(2B \backslash B)$ is equivalent to the boundedness of $\widetilde{T} : L^{\nu}(B) \rightarrow L^{\nu}(2B \backslash B)$. Obviously, by symmetry, we can apply the same argument with respect to the second variable. It shows that we can freely assume the kernel $\lambda(x,y)^{-1}$ to satisfy the Lipschitz regularity estimate \eqref{lipschitzreg} in both variables, which we will do in the following, forgetting this operator $\widetilde{T}$. Observe however that, under this assumption, we can no longer use the fact that $\lambda(x,y) = \mu(B(x,\rho(x,y)))$, only that these two quantities are comparable.

\bigskip

\noindent $(2)$ The second step in the argument now consists in applying a standard Calder\'on-Zygmund decomposition. We show that $T$ is of weak type $(1,1)$ : we prove that for all $f \in L^1(B)$, with $\supp f \subset B$
\[ \mu( \{ x \in 2B \backslash B \mid |Tf(x)| > \alpha \} ) \lesssim \frac{1}{\alpha} \|f\|_{L^1(B)}. \]
The idea is to write a Calder\'on-Zygmund decomposition of $f$ on $X$. However, we have to be a bit careful, because if we write the standard decomposition $f = g + \sum_{i \in I} b_i$ directly on the ball $B$, $g$ will not be supported inside $B$. To avoid this problem, let us use a Whitney partition of the ball $B$ : consider the dyadic cubes $Q \subset B$ which are maximal for the relation $l(Q) \leq \rho (Q, B^c)$. Call them $Q_j$, $j \in J$. They are mutually disjoint and they realize a partition of the ball $B$ but for a set of measure zero. Now, for $f \in L^1(B)$, with $\supp f \subset B$, we have $f = \sum_j f 1_{Q_j} \ \ \mu \, a. \, e.$ Set $f_j = f 1_{Q_j}$, and for every fixed $j$, write a Calder\'on-Zygmund decomposition of $f_j$ on $Q_j$: $f_j = g_j + \sum_{i \in I} b_{i,j}$ with
\begin{itemize}
\item[$\bullet$] $g_j = 1_{\Omega_{j,\alpha}^c} f_j$ where $\Omega_{j,\alpha} = \{ x \in Q_j \mid |f_j^{\ast}(x)| > \alpha \}$, $f_j^{\ast}$ denoting the dyadic maximal function of $f_j$ on $Q_j$. Thus $\supp g_j \subset \overline{Q_j}$, $g_j \in L^{\infty}$, $\|g_j \|_{\infty} \leq \alpha$, and $\|g_j\|_p \leq (\|g_j\|_{\infty}^{p-1} \|g_j\|_1)^{1/p} \leq \alpha^{1/{p'}} \|f_j\|_1^{1/p}$.
\item[$\bullet$] $\supp b_{i,j} \subset \overline{Q_{i,j}}$ where the sets $Q_{i,j}$ are dyadic subcubes of $Q_j$ of center $z_{Q_{i,j}}$, realizing in turn a Whitney partition of the open set $\Omega_{j,\alpha}$ : $\mu(\Omega_{j,\alpha} \backslash  \cup_i Q_{i,j}) =0$, $Q_{i,j}$ are mutually disjoint, and there exists a dimensional constant $C>C_1$ (where $C_1$ is the dimensional constant of Lemma \ref{cubes}) such that for every $i$, $B(z_{Q_{i,j}}, C l(Q_{i,j})) \cap \Omega_{j,\alpha}^c \neq \varnothing$. In addition, we have $[|b_{i,j}|]_{Q_{i,j}} \lesssim \alpha$ and $[b_{i,j}]_{Q_{i,j}} =0$.
\item[$\bullet$] $\|f_j\|_1 = \sum_i{\|b_{i,j}\|_1} + \|g_j\|_1$.
\end{itemize}
\smallskip
\noindent Now, set $g= \sum_j g_j$. Observe that we have $\supp g \subset B$, and $f = g + \sum_{i,j} b_{i,j}$. Applying \eqref{3etoile} for the couple $(p,p')$, and the disjointness of the dyadic cubes $Q_j$, we have
\begin{align*}
\mu(\{ x \in 2B \backslash B \mid |Tg(x)| > \alpha/2 \}) & \lesssim \frac{1}{\alpha^p} \int_{2B\backslash B}{|Tg|^p \dd \mu} \lesssim \frac{1}{\alpha^p} \int_{B}{|g|^p \dd \mu}\\
&  =  \frac{1}{\alpha^p}  \sum_j \int_{Q_j}{|g_j|^p \dd \mu}    =  \frac{1}{\alpha^p} \sum_j \|g_j\|^p_p \\
& \leq \frac{1}{\alpha} \sum_{j} \|f_j\|_{1}= \frac{1}{\alpha} \|f\|_{L^1(B)}.
\end{align*}
Also,
\begin{align*}
\mu  (  \{ x \in 2B \backslash B \mid   |T  (\sum_{i,j} b_{i,j}  )(x) | \! > \! \alpha/2  \} ) \! &  \leq \! \mu(\bigcup_{i,j} \widehat{Q_{i,j}}) \! + \!  \mu \Big ( (\bigcup_{i,j} \widehat{Q_{i,j}})^c \cap \{ x \in X \! \mid \!\! \sum_{i,j} | T b_{i,j} (x)| \! > \! \alpha/2 \}\Big)\\
& \lesssim \sum_{i,j}{\mu(Q_{i,j})}  + \mu(\{ x \in X \mid \sum_{i,j} |T b_{i,j}(x)| 1_{\widehat{Q_{i,j}}^c} > \alpha/2  \})\\
& \lesssim \sum_{i,j}{\mu(Q_{i,j})}  + \sum_{i,j} \frac{1}{\alpha} \int_{\widehat{Q_{i,j}}^c} {|T b_{i,j}| \dd \mu}.
\end{align*}
Since $b_{i,j}$ is of mean $0$ on $Q_{i,j}$, and because the kernel $\lambda(x,y)^{-1}$ satisfies the Hölder standard estimate, we can apply the standard Calder\'on-Zygmund estimates to obtain
\begin{align*}
\int_{\widehat{Q_{i,j}}^c} {|T b_{i,j}| \dd \mu}  & \lesssim \int_{\widehat{Q_{i,j}}^c}\int_{Q_{i,j}}{|b_{i,j}(x)| \left | \lambda(z_{Q_{i,j}},y)^{-1} - \lambda(x,y)^{-1}  \right | \dd \mu(x) \dd \mu(y)}\\
& \lesssim \|b_{i,j}\|_1 \int_{\widehat{Q_{i,j}}^c}{\frac{1}{\lambda(z_{Q_{i,j}},y)} \left( \frac{l(Q_{i,j})}{\rho(z_{Q_{i,j}},y)}\right ) \dd \mu(y)} \\
&\lesssim  \|b_{i,j}\|_1 \sum_{k \geq 0} \int_{ 2^k  l(Q_{i,j})  \leq \rho(y,z_{Q_{i,j}}) < 2^{k+1}  l(Q_{i,j})} { \frac{2^{-k }}{\mu(B(z_{Q_{i,j}}, 2^k  l(Q_{i,j})))} \dd \mu(y)} \\
& \lesssim  \|b_{i,j}\|_1 \sum_{k \geq 0} {2^{-k}} \lesssim \alpha \mu(Q_{i,j}) .
\end{align*}
Finally, we have
\begin{align*}
\mu  (  \{ x \in 2B \backslash B \mid   |T  (\sum_{i,j} b_{i,j}  ) (x)| > \alpha/2  \} ) & \lesssim \sum_{i,j} \mu(Q_{i,j}) \lesssim \sum_j \mu(\Omega_{j,\alpha}) \lesssim \sum_j \frac{1}{\alpha} \|f_j\|_{1} = \frac{1}{\alpha} \|f\|_{L^1(B)}.
\end{align*}
Thus, we get as expected
\[ \mu(\{ x \in 2B \backslash B \mid |Tf(x)| > \alpha \}) \lesssim \frac{1}{\alpha} \|f\|_{L^1(B)}. \]
By interpolation, it follows that we have $Tf \in L^q(2B \backslash B)$ for every $1<q \leq p$ and $f \in L^q(B)$, with $\supp f \subset B$.

\bigskip

\noindent $(3)$ The last part of our argument consists in applying some duality argument to conclude. \eqref{3etoile} for the couple $(p,p')$ implies that the adjoint operator $T^{\ast}$ is bounded from $L^{p'}(2B \backslash B)$ to $L^{p'}(B)$. We show again that $T^{\ast}$ is of weak type $(1,1)$: we prove that for all $f \in L^1(2B \backslash B)$, with $\supp f \subset 2B \backslash B$,
\[ \mu( \{ x \in B \mid |T^{\ast}f(x)| > \alpha \} ) \lesssim \frac{1}{\alpha} \|f\|_{L^1(2B \backslash B)}. \]
The idea is to use as before a Whitney partition of the open set $\overline{B}^c$: consider the dyadic cubes $Q_j \subset \overline{B}^c$ that are maximal for the relation $l(Q) \leq \rho(Q,B)$. They partition $\overline{B}^c $ but for a set of measure $0$. Of course, there are some of these $Q_j$ that intersect $(2B)^c$. To overcome this problem, let us keep only the cubes $Q_j, j \in J,$ intersecting the set $cB\backslash B$, with $1<c<\frac{C_1 +2}{C_1 +1}$. These cubes satisfy, for every $j \in J$, $l(Q_j) \leq \rho(Q_j,B) \leq (c-1)r$. Because of property $(4)$ of Christ's dyadic cubes (Lemma \ref{cubes}), it implies that for every $x \in Q_j$, $\rho(x, z_B) \leq cr + \diam Q_j \leq cr + C_1 (c-1)r <2r$, so that $Q_j \subset 2B \backslash B$. Now, for $f \in L^1(2B \backslash B)$, with $\supp f \subset 2B \backslash B$, write
\[ f = \sum_{j \in J} f 1_{Q_j} + f 1_{(2B\backslash B) \backslash \cup_{j \in J} Q_j} = \sum_{j \in J} f_j + \tau \quad \mu \, a.e.  \]
Remark that this is where we use the assumption $\mu(\bar{B}\backslash B) =0$, and it is the only time that we use it in our proof. Apply the same argument as in step $(2)$ to $f_j$, for every $j \in J$, to get
\[ \mu( \{ x \in B \mid |T^{\ast} (\sum_{j\in J} f_j)(x)| > \alpha/2 \} ) \lesssim \frac{1}{\alpha} \|f\|_{L^1(2B \backslash B)}. \]
For the remaining term $\tau$, observe that since $\Big ( (2B\backslash B) \backslash \cup_{j \in J} Q_j \Big ) \subset (cB)^c$, $\supp \tau \subset (cB)^c $, and we trivially have $T^{\ast} \tau \in L^1(B)$. Indeed, by the doubling property, we have
\[ \int_B |T^{\ast} \tau| \dd \mu \leq \int_B \int_{(cB)^c} \frac{|\tau(y)|}{\lambda(x,y)} \dd \mu(y) \dd \mu(x) \lesssim \| \tau \|_1 \lesssim \|f\|_{L^1(2B \backslash B)}. \]
Hence,
\[ \mu( \{ x \in B \mid |T^{\ast}\tau(x)| > \alpha \} ) \lesssim \frac{1}{\alpha} \|f\|_{L^1(2B \backslash B)}, \]
and $T^{\ast}$ is of weak type $(1,1)$. By interpolation, it follows that $T^{\ast} f \in L^q(B)$ for every $1<q \leq p'$ and $f \in L^q(2B \backslash B)$, with $\supp f \subset 2B\backslash B$. By duality, the Hardy property $(\mathrm{HP})$ is satisfied on $X$.
\end{proof}

\bigskip

\begin{remark}
Observe that the first step in our argument does not directly extend to the quasi-metric setting. However, R. Macias and S. Segovia \cite{MaSe} proved that if $\rho$ is a quasi-distance on $X$, there exists another quasi-distance $\rho'$ equivalent to $\rho$ (in the sense that $\rho(x,y) \eqsim \rho'(x,y)$ for every $x,y \in X$) which is Hölder regular. See also \cite{PalStem}  for an elegant proof of this result. To adapt our proof, one only has to define the new kernel $\widetilde{\lambda}$ of the first step of the argument with this quasi-metric $\rho'$ instead of $\rho$. The estimate \eqref{lipschitzreg} is then replaced by a Hölder regularity estimate, where the $\rho'(x,x'), \rho'(x,y)$ that appear can be replaced by $\rho(x,x'), \rho(x,y)$ because $\rho \eqsim \rho'$. For the next steps in the argument, one again assumes the kernel to satisfy the Hölder regularity estimate obtained and works as in the metric setting with the quasi-metric $\rho$, forgetting $\rho'$. Note that the existence of Christ's dyadic cubes given by Lemma \ref{cubes} remains valid in the quasi-metric setting (see \cite{Christ}), allowing us to do the same kind of Whitney coverings.
\end{remark}

%
%
\bigskip

\section{Sufficient conditions for the Hardy property}

\bigskip

\subsection{Layer decay properties \eqref{OLD}, \eqref{LUOLD}}

It is not clear when \eqref{3etoile} is true in a space of homogeneous type for a given ball $B$. In fact, it is false in general, as will be illustrated by some counterexamples in the following sections. It obviously depends on how the sets $B$ and $B^c$ see each other in $X$. By analogy with Christ's dyadic cubes, natural objects are the outer and inner layers $\{ x \in B | \rho(x,B^c) \leq \varepsilon \}$ and $\{ y \in B^c  | \rho(y,B) \leq \varepsilon  \}$. We shall assume they tend to zero in measure as $\varepsilon \rightarrow 0$, and in a scale invariant way, as expressed by the following definition.

\bigskip

\begin{definition}\label{oldluold}{Layer decay and relative layer decay properties.\\}
Let $(X,\rho,\mu)$ be a space of homogeneous type. For a ball $B$ in $X$, set $B_{\varepsilon} = \{ x \in B  | \rho(x,B^c) \leq \varepsilon  \} \cup \{ y \in B^c  | \rho(y,B) \leq \varepsilon  \}$ the union of the inner and outer layers.
\begin{itemize}
\item We say that $X$ has the \textbf{layer decay property} if there exist constants $\eta>0$, $C<+\infty$ such that for every ball $B=B(z,r)$ in $X$ and every $\varepsilon >0$, we have
\begin{equation} \label{OLD} \tag{LD}
\mu(B_{\varepsilon}) \leq C \left (\frac{\varepsilon}{r} \right )^{\eta} \mu(B(z,r)).
\end{equation}
\item We say that $X$ has the \textbf{relative layer decay property} if there exist constants $\eta>0$, $C<+\infty$ such that for every ball $B=B(z,r)$ in $X$, every ball $B(w,R)$ with $R \leq 2r$, and every $\varepsilon >0$, we have
\begin{equation} \label{LUOLD} \tag{RLD}
\mu\left (B_{\varepsilon}\cap B(w,R) \right ) \leq C \left (\frac{\varepsilon}{R} \right )^{\eta} \mu(B(w,R)).
\end{equation}
\end{itemize}
\end{definition}

\smallskip

\noindent This is Definition $9.1$ of \cite{AR}, with a minor modification to the relative layer decay property where we have substituted the condition $R \leq 2r$ to $z \notin B(w,R)$. We think this is a better definition, though equivalent, because this property is only relevant for small $R$ (else it says nothing), and when $R$ is small enough compared to $r$, if $B_{\varepsilon}\cap B(w,R)\neq \varnothing$ then necessarily $z \notin B(w,R)$.

\bigskip

The layer decay property already appeared in \cite{Buck2} (with only $\mu\left (\{ x \in B  | \rho(x,B^c) \leq \varepsilon  \}  \right )$ in the left hand side of \eqref{OLD}). These properties express the fact that the points are distributed in such a way in the homogeneous space that they never concentrate too much in the inner and outer layers of balls. One can note that while \eqref{OLD} is a global condition, a sort of averaging property over the whole layer, \eqref{LUOLD}  is a rather local condition. Remark also that \eqref{LUOLD} trivially implies \eqref{OLD}. Because of the opposed local and global nature of \eqref{LUOLD} and \eqref{OLD} though, it is sensible to think that these properties should not be equivalent. We will prove it in Section $7.2$.
%

\medskip

It turns out that the relative layer decay property constitutes a sufficient condition for the Hardy property $(\mathrm{HP})$:

\begin{prop}\label{LUOLDP}
Let $(X,\rho,\mu)$ be a space of homogeneous type, and suppose that $X$ has the relative layer decay property \eqref{LUOLD}. Then the Hardy property $(\mathrm{HP})$ is satisfied on $X$.
\end{prop}

\noindent This is Proposition $9.2$ of \cite{AR}. To shed some light on what will come next, let us recall the proof here. The argument is quite similar to the one we used in the dyadic setting (see the proof of Lemma $2.4$ of \cite{AR}), with adaptations owing to the fact one cannot use exact coverings with balls as for dyadic cubes. 

\begin{proof}
 Fix a ball $B$ of center $z$ and radius $r>0$, and functions $f,g$ respectively supported on $2B\backslash B$ and $B$ with $f \in L^{\nu}$, $g \in L^{\nu'}$. Let $I$ be the integral in \eqref{3etoile}. For a locally integrable function $f$, we will denote by $M_{\mu}(f)$ the centered maximal function
\[ M_{\mu}(f)(x) = \sup_{\tau >0} \frac{1}{\mu(B(x,\tau))} \int_{B(x,\tau)} |f(y)| \dd \mu(y). \]
Recall that by the Hardy-Littlewood maximal theorem, $M_{\mu}$ is of strong type $(p,p)$ for any $1<p<+\infty$, and of weak type $(1,1)$. We refer for example to \cite{Stein93} for further details. We prove that for all $1 < \sigma,s < \infty, $ we have
\begin{equation}\label{inthardy}
I \lesssim   \left \langle \left (M_\mu (|f|^{\sigma}) \right )^{1/\sigma} , |g| \right \rangle + \left \langle |f| , \left (M_\mu (|g|^s) \right )^{1/s} \right \rangle  ,
\end{equation}
and the Hardy-Littlewood maximal theorem then gives the desired result, choosing $1 < \sigma< \nu$ and $1 < s < \nu'$. Without loss of generality, we can assume $f,g \geq 0$. The hypotheses imply $\mu(\overline{B}\backslash B) = 0$ and allow us to write
\[ I  =  \int_{x \in B \atop \rho (x,B^c)>0}{\!\!\!\! g(x) \! \int_{y \in 2B \backslash B  \atop \rho (y,B) \leq \rho (x,B^c) }{\!\! \frac{f(y)}{\lambda(x,y)}\dd \mu(y)\dd \mu(x)}} + \int_{y \in 2B \backslash B \atop \rho (y,B)>0}{\!\! f(y) \! \int_{x \in B \atop \rho (y,B) > \rho (x,B^c)}{\!\! \frac{g(x)}{\lambda(x,y)}\dd \mu(x)\dd \mu(y)}}. \]
Let us begin by estimating the first term. Fix $x \in B$. As in the dyadic setting, let
\[ E_x = \left \{ y \in 2B \backslash B \, | \, \rho (y,B) \leq \rho (x,B^c) \right \}.\]
We decompose the integral in $y$ over coronae at distance $2^j \rho(x,B^c)$ from $x$:
\begin{align*}
 I_1(x) & =  \int_{y \in E_x }{\frac{f(y)}{\lambda(x,y)} \dd \mu(y)}  \\
 & \lesssim \ \  \ \ \sum_{j \geq 0} \ \  \ \ \int_{y \in E_x \atop 2^j \rho(x,B^c) \leq \rho(x,y) < 2^{j+1}\rho(x,B^c)}{\frac{f(y)}{\mu(B(x,\rho(x,y)))}\dd \mu(y)}\\
& \lesssim  \sum_{j\geq 0 \atop z \notin B(x, 2^{j+1}\rho(x,B^c))} M_{\mu}(f^{\sigma})(x) ^{1/\sigma} \left ( \frac{\mu(E_x \cap B(x,2^{j+1}\rho(x,B^c)))}{\mu( B(x,2^{j+1}\rho(x,B^c)))} \right )^{\frac{1}{\sigma'}}\\
& + \sum_{j \geq 0 \atop z \in B(x, 2^{j+1}\rho(x,B^c))} \frac{1}{\mu(B(x,2^{j+1}\rho(x,B^c)))} \int_{y \in E_x \atop 2^j \rho(x,B^c) \leq \rho(x,y) < 2^{j+1}\rho(x,B^c) }{f\, \dd \mu},\\
\end{align*}
where the last inequality is obtained applying the Hölder inequality with $\sigma>1$. Observe that there are at most four integers $j \geq 0$ such that $z \in B(x, 2^{j+1}\rho(x,B^c))$ and $E_x \cap \{ y \in B^c | 2^j \rho(x,B^c) \leq \rho(x,y) < 2^{j+1}\rho(x,B^c) \} \neq \varnothing$. Indeed, let $j_0$ be the first such integer, which implies $\rho(x,z) \leq 2^{j_0 +1} \rho(x,B^c)$, and let $j\geq j_0$ be another one. Let $y \in E_x \cap \{ y \in B^c | 2^j \rho(x,B^c) \leq \rho(x,y) < 2^{j+1}\rho(x,B^c) \}$. Using $y \in E_x$, we have $\rho(z,y) \leq r + \rho(y,B) \leq r + \rho(x,B^c)$. Also, $r \leq \rho(x,z)+\rho(x,B^c)$. Hence $\rho(z,y) \leq \rho(x,z) +2 \rho(x,B^c)$. Using $2^j \rho(x,B^c) \leq \rho(x,y)$, we obtain
\[ 2^j \rho(x,B^c) \leq \rho(x,y)\leq \rho(x,z) + \rho(z,y) \leq 2 \rho(x,z) +2 \rho(x,B^c) \leq  (2^{j_0 +2}+2) \rho(x,B^c), \]
hence $j \leq j_0 +3.$\\
Moreover, if $z \notin B(x, 2^{j+1}\rho(x,B^c))$, since $x \in B$, we have $2^{j+1}\rho(x,B^c) \leq 2r$. Consequently, we can apply \eqref{LUOLD} and we get
\[ I_1(x) \lesssim  M_{\mu}(f^\sigma)(x) ^{1/\sigma} \sum_{j\geq 0} {\left ( \frac{\rho(x,B^c)}{2^{j+1}\rho(x,B^c)}\right )^{\frac{\eta}{\sigma'}}} + 4M_{\mu}f(x) \lesssim M_{\mu}(f^\sigma)(x) ^{1/\sigma}, \]
so the first integral is controlled by $\langle g,  (M_\mu (f^{\sigma})  )^{1/\sigma} \rangle$. The argument for the second integral is entirely similar using the symmetry of our assumptions (and $4$ above becomes $3$). This proves \eqref{inthardy}.
\end{proof}

\bigskip

\begin{remarks}\begin{enumerate}
\item Obviously, if $X$ satisfies the layer decay property \eqref{OLD}, then we have $\mu(\bar{B} \backslash B) =0$ for every ball $B \subset X$. Thus, with Proposition \ref{LUOLDP} and Proposition \ref{Hardy1couple}, we have proved that if $X$ satisfies \eqref{LUOLD}, then $X$ satisfies $(\mathrm{HP})$, and if $X$ satisfies only \eqref{OLD}, then $X$ satisfies $(\mathrm{HP})$ provided it satisfies the Hardy inequality \eqref{3etoile} for one \mbox{particular} couple of exponents $(p,p')$.
\item Note that if $f$ and $g$ are taken in $L^{\nu_1}(X)$ and $L^{\nu_2}(X)$ with $1/{\nu_1}+1/{\nu_2} <1$ (which is stronger than $g \in L^{\nu_{1}'}$), then if $I$ denotes the integral in \eqref{3etoile}, the normalised inequality
\[  \frac{1}{\mu(B)} I   \lesssim \|f\|_{L^{\nu_1}\left (2B \backslash B, \frac{\dd \mu}{\mu(B)}\right )}  \|g\|_{L^{\nu_2}\left (B, \frac{\dd \mu}{\mu(B)}\right )}   \]
is true in any space of homogeneous type only satisfying \eqref{OLD}. The proof is in this case much easier and we refer to \cite{AR} for the detail.
\end{enumerate}
\end{remarks}

\bigskip

\subsection{Annular decay properties \eqref{AD}, \eqref{LUAD}}

Interestingly, another very similar condition already appeared in the literature (see for example \cite{MaSe}, \cite{DJS}, \cite{Buck}, \cite{Tess}). It is the notion of annular decay that we recall now.

\begin{definition}{Annular decay and relative annular decay properties.\\}
Let  $(X,\rho,\mu)$ be a space of homogeneous type. For $z \in X$ and $r>0,0<s<r$, set $C_{r,r-s}(z) = B(z,r) \backslash B(z,r-s)$.
\begin{itemize}
\item We say that $X$ has the \textbf{annular decay property} if there exist constants $\eta>0$ and $C<+\infty$ such that for every $z \in X, r>0,0<s<r$, we have
\begin{equation} \label{AD} \tag{AD}
\mu(C_{r,r-s}(z)) \leq C \left (\frac{s}{r} \right )^{\eta} \mu(B(z,r)).
\end{equation}
\item We say that $X$ has the \textbf{relative annular decay property} if there exist constants $\eta>0$ and $C<+\infty$ such that for every $z \in X$, $r>0,0<s<r$, and every ball $B(w,R)$ with $R \leq 2r$, we have
\begin{equation} \label{LUAD} \tag{RAD}
\mu ( C_{r,r-s}(z) \cap B(w,R) ) \leq C \left (\frac{s}{R} \right )^{\eta} \mu(B(w,R)).
\end{equation}
\end{itemize}
\end{definition}
\noindent Once again, this is Definition $9.4$ of \cite{AR}, with the same minor modification in the definition of the relative annular decay property than before. Note that this condition \eqref{AD} was made an assumption in \cite{DJS} for the first proof of the global $Tb$ theorem in a space of homogeneous type. Again, \eqref{AD} is a global property while \eqref{LUAD} is a local one. Similarly as for layer decay properties, we have that \eqref{LUAD} implies \eqref{AD}. Observe that for a ball $B=B(x,r)$, with $x \in X$, $r>0$, we have, if $\varepsilon >0$,
\[ B_{\varepsilon} = \{ y \in C_{r,r-\varepsilon}(x) \mid \rho(y,B^c) \leq \varepsilon \} \cup \{ y \in C_{r+2 \varepsilon,r}(x) \mid \rho(y,B) \leq \varepsilon \} .   \]
It follows that $B_{\varepsilon} \subset C_{r+2\varepsilon,r-\varepsilon}(x)$ and thus \eqref{LUAD} (respectively \eqref{AD}) implies \eqref{LUOLD} (respectively \eqref{OLD}). In particular, \eqref{LUAD} is a sufficient condition for the Hardy property $(\mathrm{HP})$ because of Proposition \ref{LUOLDP}.

\bigskip

\section{Geometric properties ensuring the relative layer decay property}

\bigskip

\subsection{Monotone geodesic property \eqref{monotone}}

In \cite{Buck}, Buckley introduces the notion of chain ball spaces and proves that under that condition, a doubling metric measure space satisfies \eqref{AD}. Colding and Minicozzi II already had proved that this property was satisfied by doubling complete riemannian manifolds in \cite{CM}. Tessera introduced a notion of monotone geodesic property in \cite{Tess}, and proved that this property also implies \eqref{AD} (called there the Føllner property for balls) in a doubling metric measure space. Lin, Nakai and Yang recently showed in \cite{LNY} that chain ball and a slightly stronger scale invariant version of the monotone geodesic are equivalent. It is the latter that will interest us.

\begin{definition}
Let $(X,\rho)$ be a metric space. We say that $X$ has the \textbf{monotone geodesic property} $(\mathrm{M})$ if there exists a constant $0< C<+\infty$ such that for all $u>0$ and all $x,y \in X$ with $\rho(x,y) \geq u$, there exists a point $z \in X$ such that
\begin{equation} \label{monotone}  \tag{M}
\rho(z,y) \leq Cu \quad \mathrm{and} \quad \rho(z,x) \leq \rho(y,x) - u.
\end{equation}
\end{definition}

\medskip

\noindent Remark that $C$ must satisfy $C\geq 1$. Remark also that iterating this property, one gets that for every $x,y \in X$ with $\rho(x,y) \geq u$, there exists a sequence of points $y_0=y, y_1,..., y_m=x$ such that for every $i \in \{0,...,m-1 \}$
\[ \rho(y_{i+1},y_i) \leq Cu \quad \mathrm{and} \quad \rho(y_{i+1},x) \leq \rho(y_i,x) - u.  \]

Observe that this is a purely metric property. It is obviously satisfied by complete doubling Riemannian manifolds. It is also satisfied by any geodesic space or length space (see \cite{Burago} for a definition). It appears that \eqref{monotone} not only yields the annular decay property, but also, as was proved in \cite{AR}, the stronger relative annular decay property.

\medskip

\begin{prop}\label{Delta}
Let $(X,\rho,\mu)$ be a space of homogeneous type, and suppose that $X$ has the monotone geodesic property \eqref{monotone}. Then $X$ has the relative annular decay property \eqref{LUAD}.
\end{prop}

\medskip

\noindent We refer again to \cite{AR} (Proposition $9.6$) for the proof of this result (our modification on \eqref{LUAD} has no impact). The argument essentially adapts the one in \cite{Tess} with more care on localization.

\bigskip

\begin{remark}
Observe that conversely, neither $(\mathrm{HP})$ nor \eqref{LUAD} imply \eqref{monotone}. Let us give two examples to illustrate this. First consider the space formed by the real line from which an arbitrary interval has been withdrawn, equipped with the Euclidean distance and Lebesgue measure. This space obviously does not have the monotone geodesic property, as, to put it roughly, there is a hole in it. On the other hand, this space clearly satisfies the Hardy property, as a consequence of $(\mathrm{HP})$ on the real line, as well as \eqref{LUAD}. The second example is a connected one: consider the space made of the three edges of an arbitrary triangle in the plane, again equipped with the induced Euclidean distance and Lebesgue measure. This space has the Hardy property, once again as a straightforward consequence of the fact that the unit circle has it and easy change of variables. It easily follows from the fact that one of the angles must be less than $\pi/2$ that it does not have the monotone geodesic property: one of the pairs $(x,y)$ with $x$ a vertex and $y$ its orthogonal projection on the opposite side cannot meet condition \eqref{monotone}. In passing, it proves that this property is not stable under bi-Lipschitz mappings (see also \cite{Tess}).
\end{remark}
\bigskip

\subsection{Homogeneous balls property $(\mathrm{HB})$}Carefully looking into the proof of the Hardy inequality in the dyadic setting (Lemma $2.4$ of \cite{AR}), it is easy to see that if $B$ and $X\backslash B$ are themselves spaces of homogeneous types with uniform constants, then there will be no difficulty to prove that $(\mathrm{HP})$ is satisfied, as one can then adapt the proof using Christ's dyadic cubes both on $B$ and $X \backslash B$. This motivates the following definition.

\begin{definition}
Let $(X,\rho,\mu)$ be a space of homogeneous type. Let $B$ be a ball in $X$, and suppose that $B$ and $B^c$ are themselves spaces of homogeneous type with doubling constant $C_B$, \emph{i.e.} for all $x \in B$, $y \in B^c$, and all $r>0$, we have
\[ \mu(B(x,2r) \cap B) \leq C_B \ \ \mu(B(x,r) \cap B), \quad \quad \mu(B(y,2r) \cap B^c) \leq C_B \ \ \mu(B(y,r) \cap B^c). \]
We say that $X$ has the \textbf{homogeneous balls property} $(\mathrm{HB})$ if this is satisfied by all the balls in $X$ and if \[ \sup_{B\subset X}{C_B} < +\infty. \]
\end{definition}

\medskip

\begin{prop}
Let $(X,\rho,\mu)$ be a space of homogeneous type.
\begin{enumerate}
\item If $X$ has the homogeneous balls property $(\mathrm{HB})$, then $X$ has the relative layer decay property $\eqref{LUOLD}$.
\item If $X$ has the homogeneous balls property $(\mathrm{HB})$, then $X$ has the Hardy property $(\mathrm{HP})$.
\end{enumerate}
\end{prop}

\begin{proof}
$(1)$ If $X$ has the homogeneous balls property, then every ball in $X$ as well as its complement in $X$ can be partitioned into dyadic cubes, with uniform constants (see \cite{Christ}). But these cubes themselves have the layer decay property, and it is easy to see that this property transposes to the balls. Let $x \in X$, $r >0$, $B=B(x,r)$, $w \in X$, $0<R\leq 2r$, and fix $\varepsilon >0$. Let us estimate for example the measure of the inner layer $C_{\varepsilon} = \{ x \in B: \rho(x,B^c) \leq \varepsilon \}$. We want to prove that there exists $\eta' >0$ such that $\mu(C_{\varepsilon} \cap B(w,R))) \lesssim \left ( \frac{\varepsilon}{R} \right )^{\eta'} \mu(B(w,R))$. As $B$ constitutes by itself a space of homogeneous type with uniform doubling constant, there exists at every scale a partitioning of $B$ into Christ's dyadic cubes (with uniform constants for these cubes). The idea is to pave $B$ by dyadic cubes of a well chosen generation. Let $N,m,l \in \Z$ be such that $ \delta^{N+1} < r \leq \delta^N$, $\delta^{N+m+1} < \varepsilon \leq \delta^{N+m}$, and $\delta^{N+l+1} < R \leq \delta^{N+l}$. We can assume that $\varepsilon \leq R$, since otherwise the result is trivial, which means $m\geq l$. We look at the cubes of generation $j= N + [\frac{m+l}{2}] \geq N$ because $R\leq 2r$ (otherwise there would be no such cubes): for all $y \in C_{\varepsilon}$, there exists a unique $\beta$ such that $y \in Q_{\beta}^{j}$, and $\rho(y, (Q_{\beta}^{j})^c) \leq \rho(y,B^c) \leq \varepsilon  \leq \delta^{m+N}.$ By the small boundary property \eqref{petitefrontiere} of Christ's dyadic cubes, one gets that for all $\beta$
\[ \mu(C_{\varepsilon} \cap Q_{\beta}^j ) \lesssim \left ( \frac{\delta^{N+m}}{\delta^{j}} \right )^{\eta} \mu(Q_{\beta}^j) \lesssim \left (\frac{\varepsilon}{R} \right )^{\eta/2} \mu(Q_{\beta}^j). \]
On the other hand, if $C_{\varepsilon} \cap Q_{\beta}^j  \cap B(w,R) \neq \varnothing$ and if $y \in C_{\varepsilon} \cap Q_{\beta}^j  \cap B(w,R) $, $x \in Q_{\beta}^j $, then, using the fact that since $m \geq l$, $\delta^j \leq \delta^{N+l} \lesssim R$, we have
\[ \rho(x,w) \leq \rho(x,y) + \rho(y,w) \leq C_1 \delta^j +R \leq CR  \]
for a dimensional constant $C>0$, and then $Q_{\beta}^j \subset B(w,CR)$. Finally, we get
\begin{align*}
\mu(C_{\varepsilon} \cap B(w,R) ) & = \sum_{\beta : C_{\varepsilon} \cap Q_{\beta}^{j} \cap B(w,R) \neq \varnothing } { \mu(C_{\varepsilon} \cap Q_{\beta}^{j} \cap B(w,R) ) }\\
& \lesssim  \sum_{\beta :  C_{\varepsilon} \cap Q_{\beta}^{j} \cap B(w,R) \neq \varnothing} \left ( \frac{\varepsilon}{R} \right )^{\eta/2} \mu(Q_{\beta}^{j})\\
& \lesssim \left ( \frac{\varepsilon}{R} \right )^{\eta /2} \mu(B(w,CR)) \lesssim \left ( \frac{\varepsilon}{R} \right )^{\eta /2} \mu(B(w,R)),\\
\end{align*}
where the last line is obtained using the disjointness of the cubes $Q_{\beta}^j$ and then the doubling property. We can do the same for the outer layer as $B^c$ also constitutes a space of homogeneous type. It proves \eqref{LUOLD}.

\medskip

\noindent $(2)$ It is a direct consequence of $(1)$ and Proposition \ref{LUOLDP}. But it can also be proved directly slightly adapting the proof of Lemma $2.4$ of  \cite{AR}. Indeed, observe that the homogeneous balls property allows to use exact coverings of $B$ and $X \backslash B$ by Christ's dyadic cubes as above, and then everything works out as in the dyadic setting, taking care of those "large" cubes which are not contained in $2B \backslash B$ in a simple manner.
\end{proof}

\bigskip

It is easy to see that a space of homogeneous type does not satisfy the homogeneous balls property $\mathrm{(HB)}$ in general. Let us give a counterexample.

\begin{example}\label{connexe}
Consider the real line, from which one has withdrawn the interval $I_{\varepsilon} = ] 1 - \varepsilon, 1 - \varepsilon^2 [,$ with $\varepsilon$ a fixed small constant. Consider the ball in this set of center $1/2$ and of radius $1/2$. It is easy to see that it has, as a space of homogeneous type, a doubling constant of at least $1/\varepsilon$: inside this ball, consider the ball of center $1-\varepsilon^2$ and radius $\varepsilon - \varepsilon^2$, and its concentric double. Now, set
\[ I_{n,\varepsilon} = ] n - \varepsilon/{2^{n-1}}, n - (\varepsilon/{2^{n-1}})^2 [, \]
and consider the space $X = \R \backslash \cup_{n \geq 1} I_{n,\varepsilon}$, equipped with the Euclidean distance and the Lebesgue measure. It is clear that $X$ does not satisfy the homogeneous balls property since the doubling constants explode. Neither does $X$ satisfy the monotone geodesic property \eqref{monotone} as, to put it roughly, it has holes in it. However, observe that $X$ satisfies both \eqref{LUOLD} and $(\mathrm{HP})$. As a matter of fact, proving $(\mathrm{HP})$ on $X$ is exactly the same as proving it for the real line, and this is trivial (it is the same for \eqref{LUOLD}). In particular, this example shows that neither the monotone geodesic property nor the homogeneous balls property are necessary conditions for $(\mathrm{HP})$.
\end{example}

\bigskip

\begin{remark}We have given two sufficient conditions for \eqref{LUOLD}, one which is purely metric, the monotone geodesic property, while the other is rather a measure property. Let us examine how these two properties are connected. It is clear that $(\mathrm{HB})$ does not imply \eqref{monotone}, as is shown by the trivial example of the real line from which an arbitrary interval has been withdrawn. Now, if we suppose that $X$ satisfies \eqref{monotone}, we have a partial result regarding the homogeneous balls property. As a matter of fact, let $B = B(x,r_0)$, and $y \in B$. We prove that for all $r >0$,
\[ \mu(B(y,2r) \cap B  ) \lesssim \mu(B(y,r) \cap B).  \]
First, if $r \geq 2r_0$, then $B \subset B(y,r) \subset B(y,2r)$ and $B(y,2r) \cap B = B(y,r) \cap B = B.$ Then, if $r \leq 2r_0$ and $\rho(x,y) \leq r_0 /2$, we have $B(y,r/4) \subset (B\cap B(y,r))$ and by the doubling property, $\mu(B(y,2r)) \lesssim \mu(B(y,r/4)) \lesssim \mu(B(y,r)\cap B)$. It only remains to study the case $r \leq 2r_0$ and $r_0 /2 < \rho(x,y) < r_0$. Let $\alpha < \min(\frac{2}{1+C}, \frac{1}{2})$, and $\beta< \alpha/2$.
We have $\rho(x,y) \geq r_0 /2 \geq r \alpha/2$ and by \eqref{monotone}, there exists a point $z$ in $X$ such that
\[ \rho(y,z) \leq C r \frac{\alpha}{2} \quad \mathrm{and} \quad \rho(x,z) \leq \rho(y,x) - r \frac{\alpha}{2}. \]
Consider the ball $B(z,\beta r)$. If $w \in B(z,\beta r),$ then
\[ \rho(w,x) \leq \rho(w,z) + \rho(z,x) \leq \beta r + \rho(y,x) - r \frac{\alpha}{2} \leq r_0. \]
Thus $B(z,\beta r) \subset B.$ Furthermore,
\[ \rho(w,y) \leq \rho(w,z) + \rho(z,y) \leq \beta r + C r \frac{\alpha}{2} < r, \]
since $\alpha < \frac{2}{1+C}$. Thus $B(z,\beta r) \subset B(y,r).$
Finally,
\[ \mu(B(y,2r) \cap B) \lesssim \mu(B(y,r)) \lesssim \mu(B(z,2r)) \lesssim \mu(B(z,\beta r)) \lesssim \mu(B(y,r)\cap B). \]
This means that every ball in $X$ is itself a space of homogeneous type, with uniform constant. However, we cannot obtain the same result for the complement of balls, and $(\mathrm{HB})$ cannot be inferred.
\end{remark}

\bigskip

With the results of these first sections, we have already proved a part of Theorem \ref{diagram}: we have proved the following

\[ \xymatrix{
 & & (\mathrm{HB}) \ar@{=>}@<1ex>[d]_{\,}  \ar@{=>}@<0,5ex>[rd]^{\,}   \\
(\mathrm{M}) \, \, \, \, \not \!\!\!\!\!\!\!  \ar@{=>}@<1,5ex>[r]  & \eqref{LUAD} \ar@{=>}[l]  \ar@{=>}[r]  \ar@{=>}[d]  & \eqref{LUOLD}  \ar@{=>}[d] \ar@{=>}[r] \ar@{=>}@<1ex>[u]|@{/}  & (\mathrm{HP}) \ar@{=>}@<-2ex>_{\,}[lu]|@{/} \\
 &  \eqref{AD}  \ar@{=>}[r] & \eqref{OLD} &
}\]

\medskip

\noindent In particular, all positive implications have been established and we look now for the negative ones.

\bigskip

\section{An example in the complex plane: the curve of Tessera}

\bigskip

To further understand these properties of the homogeneous space, we will consider in this section a curve introduced by Tessera in \cite{Tess}. This curve is given by a stairway-like construction in the complex plane, starting from $0$, and containing for every $k \in \N$ a half-circle of center $0$ and radius $2^k$. More precisely, consider in the complex plane the parametric curve $\gamma(t)$ defined for $t \geq 0$, and constructed as follows with $|\gamma'(t)| =1$ for every $t \geq 0$ : \begin{itemize}
\item[$\bullet$] $\gamma(0) = 0$.
\item[$\bullet$] $\{ \gamma(t) \mid 0 \leq t \leq t_1 \}$ is the segment $[0,1]$.
\item[$\bullet$] $\{   \gamma(t) \mid t_1 \leq t \leq t_2  \}$ is the half-circle of center $0$ and radius $1$ in the half-plane $\{\Im z \geq 0\}$.
\item[$\bullet$] By induction, for $k\geq 1$, $\{ \gamma(t) \mid t_{2k} \leq t \leq t_{2k+1} \}$ is the segment $[2^{k-1},2^k]$ or $[-2^k,-2^{k-1}]$ depending on the parity of $k$, and $\{ \gamma(t) \mid t_{2k+1} \leq t \leq t_{2k+2} \}$ is the half-circle of center $0$ and radius $2^k$ in the half-plane $\{\Im z \geq 0\}$ if $k$ is even, in the half-plane $\{\Im z \leq 0\}$ if $k$ is odd.
\end{itemize}
Set $t_0 = 0$. An easy computation shows that we have for all $k \geq 1$,
\[ t_{2k} = 1 + \pi + (1+2\pi)(2^{k-1} -1)= -\pi + (1+2\pi) 2^{k-1} , \]
\[ t_{2k+1} = 1 + \pi + (1+2\pi)(2^{k-1} -1) + 2^{k-1} = - \pi + (2+2\pi)2^{k-1}.\]
\medskip

\noindent Set $X_T = \{ \gamma(t) \mid t \geq 0 \}$, equipped with the Euclidean distance $\dd$ in $\C$ and the Hausdorff length $\Lambda$. See Figure \ref{FigproofRLDTessera} for a representation of $X_T$. For $x,y \in X_T$, denote by $(x,y)$ the arc in $X_T$ between $x$ and $y$. For $z \in X_T$, and $r>0$, we denote by $B(z,r)$ the open ball of center $z$ and radius $r$ in $X_T$ : $B(z,r) = \{ x \in X_T \mid |z-x| < r \}$. We recall that a bounded set $E \subset \C$ is said to be Ahlfors-David regular (of dimension $1$) when there exists a constant $0< C <+\infty$ such that for every $z \in \C$ and $r >0$, $\frac{1}{C} r \leq \Lambda(E \cap B(x,r)) \leq Cr$ (see \cite{Ahl}, \cite{David-LN}).

\begin{prop}\begin{enumerate}
\item $(X_T,\dd, \Lambda)$ is an Ahlfors-David set of dimension $1$, and thus $(X_T,\dd,\Lambda)$ can be seen as a space of homogeneous type.
\item $X_T$ does not satisfy the annular decay property \eqref{AD} (nor the relative annular decay property \eqref{LUAD}).
\item $X_T$ does not satisfy the homogeneous balls property $(\mathrm{HB})$, nor the monotone geodesic property \eqref{monotone}.
\item $X_T$ satisfies the relative layer decay property \eqref{LUOLD}.
\item $X_T$ satisfies the Hardy property $(\mathrm{HP})$.
\end{enumerate}
\end{prop}

\medskip

\begin{proof}
$(1)$ Let us first make a preliminary observation: as $|\gamma'(t)| = 1$ for every $t\geq 0$, we have $\Lambda((\gamma(\alpha),\gamma(\beta))) = |\alpha - \beta|$. But there exists a dimensional constant $1<C_{AD}<+\infty$ such that for all $\alpha,\beta>0$
\begin{equation}\label{eth}
\frac{1}{C_{AD}} \,   \dd(\gamma(\alpha),\gamma(\beta)) \leq   \Lambda((\gamma(\alpha),\gamma(\beta))) \leq C_{AD} \,  \dd(\gamma(\alpha),\gamma(\beta)).
\end{equation}
The left inequality is trivial. For the right inequality, let $n,m \in \N$ such that $t_n \leq \alpha < t_{n+1}$, $t_m \leq \beta < t_{m+1}$. If $n=m$, then $\gamma(\alpha)$ and $\gamma(\beta)$ are either on the same half-circle, or on the same segment, and the result is clear. Assume that $|n-m| =1$, then one of these two points is on a segment, and the other on a connected half-circle. Suppose for example that $ t_{2k} \leq \alpha < t_{2k+1}$ and $t_{2k+1} \leq \beta < t_{2k+2}$ , so that $\gamma(\alpha)$ is on the segment $[2^{k-1}, 2^k]$ and $\gamma(\beta)$ on a half-circle of center $0$ and radius $2^k$. Let $\omega = \gamma(t_{2k+1}) = 2^k$. Set $a = \dd (\gamma(\alpha),\omega)$, $c = \dd (\gamma(\beta),\omega)$ and $b = \dd (\gamma(\alpha),\gamma(\beta))$. Applying elementary triangle geometry (see Figure \ref{FigADTessera}), and the fact that $a \leq b$, $c \leq a+b \leq 2b$, write
\begin{align*}
 \Lambda((\gamma(\alpha),\gamma(\beta)))^2 &  \lesssim  (a+c )^2  \lesssim b^2 = \dd(\gamma(\alpha),\gamma(\beta))^2.
\end{align*}

\bigskip

\begin{figure}[htbp]
	\centering
		\includegraphics[scale=0.65]{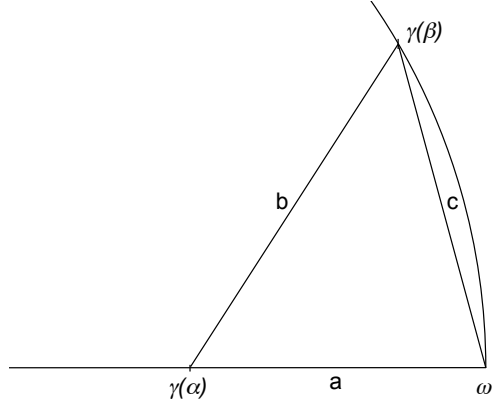}
	\caption{$X_T$ is an Ahlfors-David set of dimension $1$.}
	\label{FigADTessera}
\end{figure}

\bigskip

\noindent Finally, if $|n-m| \geq 2$, assume that for example $n\geq m$, then $\dd (\gamma(\alpha),\gamma(\beta)) \geq |\gamma(\alpha)| - |\gamma(\beta)| \geq 2^{\frac{n}{2} -1} - 2^{\frac{m}{2}}$. On the other hand, $ \Lambda((\gamma(\alpha),\gamma(\beta))) = \alpha - \beta \leq t_{n+1} - t_m \lesssim 2^{n/2}- 2^{m/2}$. The result follows. 
\noindent Now, let us check that $(X_T,\dd,\Lambda)$ is Ahlfors-David. Let $x \in X_T$, $r>0$, and set $B=B(x,r)$. As $X_T$ is connected and not bounded, the connected component of $x$ in $B$, denoted by $C(x)$, is at distance zero from the complement of $B$. Thus, there exists $y \in C(x)$ such that $\dd (x,y) \geq r/2$. But then, by \eqref{eth}, we have $\Lambda (B) \geq \Lambda((x,y)) \geq C_{AD}^{-1} \dd (x,y) \geq C_{AD}^{-1} \, (r/2) $. Moreover, set $t_0 = \inf \{ t\geq 0 \mid \gamma(t) \in B \} \geq 0$ and $t_1 = \sup \{ t\geq 0 \mid \gamma(t) \in B \} < +\infty$ because $X_T$ is unbounded. Then $\Lambda (B) \leq \Lambda ((\gamma(t_0), \gamma(t_1))) \leq C_{AD} \dd (\gamma(t_0), \gamma(t_1)) \leq C_{AD} \, 2r  $. This proves that $(X_T,\dd,\Lambda)$ is an Ahlfors-David set of dimension $1$.

\bigskip

\noindent $(2)$ As a consequence of $(1)$, observe that for any $k \geq 1$, $\Lambda(B(0,2^k)) \eqsim 2^k$. However, for every $\varepsilon >0$, we have
\[ \Lambda (B(0,2^k+ \varepsilon) \backslash B(0,2^k)) \geq \pi 2^k. \]
If $X_T$ had the annular decay property, the measure of this set would be going to $0$ with $\varepsilon$. We thus have a contradiction. Hence, $X_T$ cannot satisfy \eqref{AD}, nor \eqref{LUAD}.

\bigskip

\noindent $(3)$ It is obvious that $X_T$ does not satisfy the monotone geodesic property (pick any two points on different half circles in $X_T$). For the homogeneous balls property, fix $\varepsilon >0$ and consider the ball in $X_T$ of center $x=(0,4)$ and radius $r=3+ \varepsilon$. Let $y=(0,1)$ and $\rho = 3$. The set $B(x,r) \cap B(y,\rho)$ has only one connected component, containing $y$, and its length is comparable to $\varepsilon$. On the other hand, the set $B(x,r) \cap B(y,2 \rho)$ has two connected components, one of which containing $x$ and of length comparable to $1$. Thus, the doubling constant of the ball $B(x,r)$ seen as a space of homogeneous type exceeds $C/\varepsilon$. It follows that $X_T$ cannot satisfy $(\mathrm{HB})$, as the doubling constants of the balls cannot be uniform.

\bigskip

\noindent $(4)$ Let $0<\alpha<1 $ and let $c_1 >0$ be such that $\forall t < c_1, \ \ t^{1-\alpha} |\ln t| \leq 1 $. We first prove that $X_T$ satisfies the layer decay property \eqref{OLD}. Fix $z \in X_T, r>0, 0< \varepsilon \leq r$, set $B=B(z,r)$. Set as before $B_{\varepsilon} =  \{ x \in B  | \dd (x,B^c) \leq \varepsilon  \} \cup \{ y \in B^c  | \dd (y,B) \leq \varepsilon  \}$ the union of the inner and outer layers. Let $\theta = \varepsilon /r$. We show that there exists a dimensional constant $C< +\infty$ such that
\begin{equation} \label{resultconnec}
 \Lambda(B_{\varepsilon}) \leq C \theta^{\alpha} \Lambda(B).
\end{equation}
Observe that if $c_1 \leq \theta \leq 1$, the result is trivial:
\[ \Lambda(B_{\varepsilon}) \leq \Lambda(B) \leq \left ( \frac{1}{c_1} \right )^{\alpha} \theta^{\alpha} \Lambda(B). \]
So assume now that $\theta < c_1$. Observe that the points in $B_{\varepsilon}$ are elements of $X_T$ at distance less or equal to $\varepsilon$ from $\overline{B} \backslash B$, where $\overline{B}$ is the closure of $B$ in $X_T$: $\gamma(t) \in \overline{B} \backslash B$ if $\gamma(t) \in B^c$ and either for every $s >0$ small enough $\gamma(t+s)\in B$ or for every $s>0$ small enough $\gamma(t-s)\in B$. If $r \leq 1$, remark that $B$ and $B^c$ are connected sets, so that there are at most two points in $\overline{B} \backslash B$. But since $X_T$ is Ahlfors-David, we get $\Lambda (B_{\varepsilon}) \lesssim 2  \varepsilon$, and also $\Lambda(B) \eqsim 1$. $\eqref{resultconnec}$ follows.

\noindent We suppose now $r>1$. Assume first that $0 \notin B$, that is $|z| >r$. Denote by $C_i$, respectively $C'_j$, $\,  0 \leq i \leq p$, $0\leq j \leq p+1$, the different connected components of $B$, respectively $B^c$, starting from the one closest to the origin. Because of \eqref{eth}, it is easy to see that each $C_i$, $C'_j$ will roughly contribute to $\varepsilon$ towards $\Lambda(B_{\varepsilon})$. More precisely, set $B^i_{\varepsilon} = B_{\varepsilon} \cap (C_i \cup C'_i)$ for $0 \leq i \leq p$, and $B^{p+1}_{\varepsilon} = B_{\varepsilon} \cap C'_{p+1}$. Then, it follows from \eqref{eth} that we have for every $0\leq i \leq p+1$,
\begin{equation} \label{componentbound}
\Lambda(B^i_{\varepsilon}) \leq 4C_{AD} \, \varepsilon \lesssim  \varepsilon.
\end{equation}
Now, the idea is to estimate the number of components $C_i$, and to take care of the fact that some of them can contribute to $\Lambda(B_{\varepsilon})$ for less than $\varepsilon$, as their length can be less than that if $r$ and $\varepsilon$ are large enough. It is easy to see that $p$ can be roughly bounded by $\ln r$. Indeed, let $k_0\geq 0$ be such that $2^{k_0} \leq \dist(C_0, 0) <  2^{k_0 +1}$ (remember that we have assumed $|z| >r$), and observe (see Figure \ref{FigproofRLDTessera}) that we have, for $0 \leq i \leq p-1$, $\dd (C_i,C_{i+1}) \geq 2^{k_0 +2i} \geq 2^{2i}$. Consequently, we must have
\[ \sum_{i=0}^{p-1} 2^{2i} \leq 2r \Rightarrow 4^{p} \lesssim r \Rightarrow p \lesssim \ln r.\]
On the other hand, observe that for every $0 \leq i \leq p$, $(C_i \cup C'_i) \subset B(0,2^{k_0 +2i+2})$. Since $X_T$ is Ahlfors-David, it follows that
\begin{equation}  \label{taillecomponent}
\Lambda(C_i \cup C'_i) \lesssim 2^{k_0 + 2i +2} \lesssim 4^{k_0 +i}.
\end{equation}
Applying \eqref{componentbound}, we get $\Lambda(B^i_{\varepsilon}) \lesssim \min (\varepsilon, 4^{k_0+i})$. Finally, we obtain
\[ \Lambda(B_{\varepsilon}) = \sum_{i=0}^{p+1}{\Lambda(B^i_{\varepsilon})} \lesssim \sum_{i \geq 0 :  4^{k_0+i} \leq \varepsilon} 4^{k_0+i} + \varepsilon \, \card \{ 0 \leq i \leq p+1  \mid \varepsilon<  4^{k_0+i} \}. \]
But $\card \{ 0 \leq i \leq p+1  \mid \varepsilon<  4^{k_0+i} \} \leq \card \{ 0 \leq i \leq p+1  \mid \varepsilon<  4^{i} \}$. Remark that if $4^{i} > \varepsilon$, then $i \geq \frac{\ln\varepsilon}{\ln 4}$. The cardinal intervening in the second term is thus bounded by $C (\ln r - \ln \varepsilon)$ where $C$ is an absolute constant. Consequently, we have
\[ \Lambda(B_{\varepsilon}) \lesssim \varepsilon + \varepsilon \ln \left ( \frac{r}{\varepsilon}  \right ) \lesssim \varepsilon (1 - \ln \theta) \lesssim \theta^{\alpha} r, \]
because $\varepsilon (1 - \ln \theta) \theta^{-\alpha} r^{-1} = \theta^{1-\alpha} - \theta^{1-\alpha} \ln \theta \leq 2$ as $\theta < c_1$. But since $X_T$ is Ahlfors-David, we have $\Lambda(B) \eqsim r$ and \eqref{resultconnec} follows. \\
\noindent It remains to consider the case when $0 \in B$. But the same argument still works, only the notations have to be slightly modified because, this time, the origin belongs to the connected component $C_0$ of $B$. Thus, there is in this case the same number of components $C_i$ and $C'_i$, and it is the distance $\dd (C'_0,0)$ that plays a role in the argument instead of $\dd (C_0,0)$. Apart from this, the argument is mostly unchanged, so we do not elaborate on it here.

\begin{figure}[htbp!]
	\centering
		\includegraphics[scale=0.9]{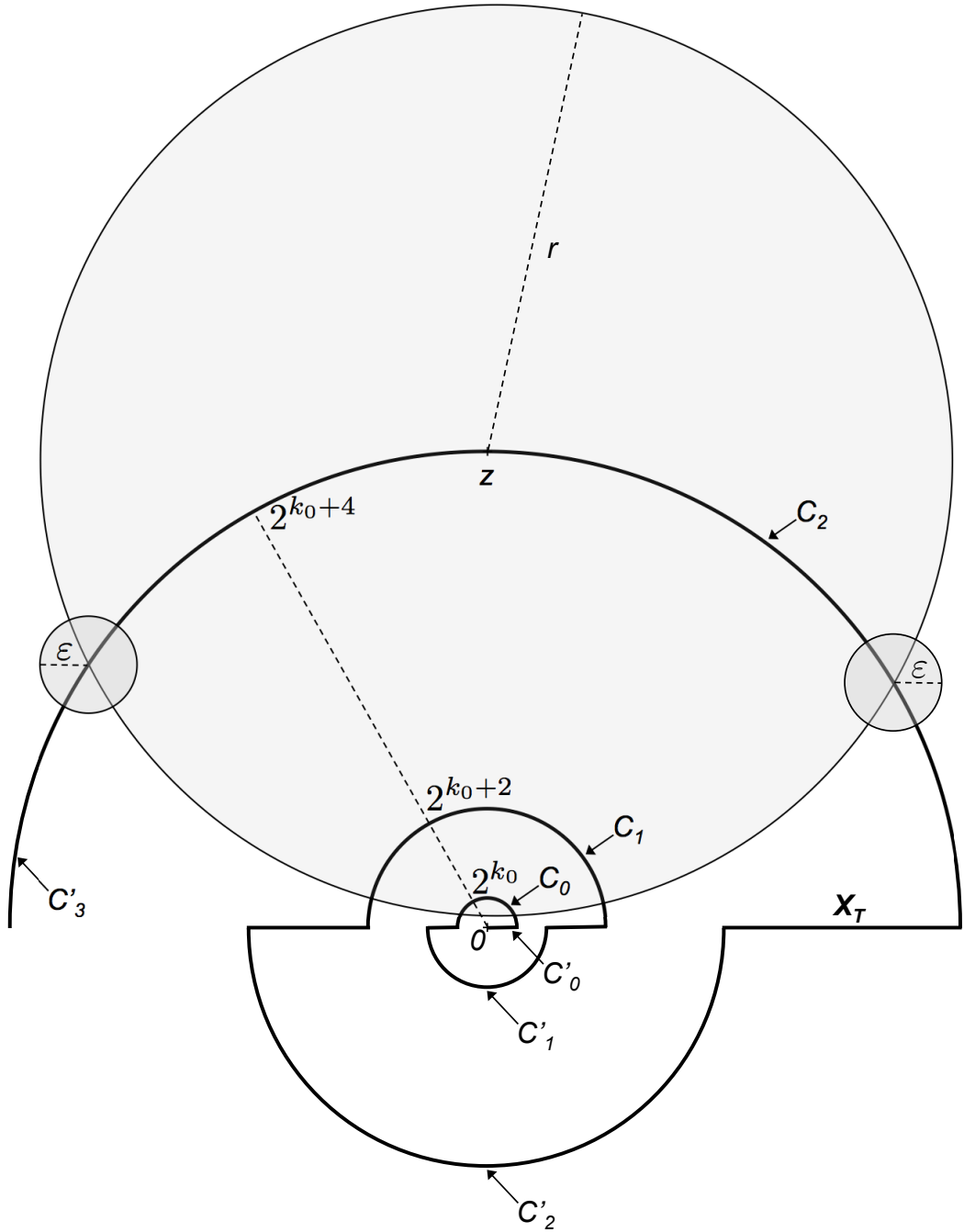}
	\caption{Layer decay property \eqref{OLD} in the space $X_T$.}
	\label{FigproofRLDTessera}
\end{figure}

It remains to prove \eqref{LUOLD}. Let $B= B(z,r)$ be a ball in $X_T$ as before. Let $0<R\leq 2r$, $w \in X_T$. We prove that
\begin{equation}\label{proofRLDTessera}
\Lambda (B_{\varepsilon} \cap B(w,R)) \lesssim \left (\frac{\varepsilon}{R} \right )^{\alpha} \Lambda(B(w,R)),
\end{equation}
Once again, when $\varepsilon /R \geq c_1$, the result is trivial, so we can assume that $\varepsilon /R <c_1$. Now, observe that we can apply exactly the same argument as above. The only difference is that instead of estimating the total number $p$ of connected components $C_i$ of $B$, we now have to estimate the number of these connected components that intersect $B(w,R)$. But by the same argument as before, this number is bounded by $(\ln R)$ as soon as $R>1$ (and the result is trivial when $R \leq 1$). Going through with the argument, this provides the bound
\[ \Lambda(B_{\varepsilon} \cap B(w,R)) \lesssim \left ( \frac{\varepsilon}{R} \right )^{\alpha} R. \]
But since $X_T$ is Ahlfors-David, we have $\Lambda (B(w,R)) \eqsim R$ and \eqref{proofRLDTessera} follows.

\bigskip

\noindent $(5)$ Applying Proposition \ref{LUOLDP}, $(5)$ is a direct consequence of $(4)$, but to better understand this example, we will give a direct proof here. Fix $z \in X_T$, $r>0$, $1<\nu<+\infty$, set $B=B(z,r)$, and let $f \in L^{\nu}(B)$, $f$ supported on $B$, $g\in L^{\nu'}(2B\backslash B)$, $g$ supported on $2B\backslash B$. Remark that because of Proposition \ref{Hardy1couple}, we could limit ourselves to the case when $\nu = \nu'= 2$, but we will keep on working with undefined exponents to show that they do not play any part in our argument and that the latter does not rely on any specific geometry brought by $L^2$ integrability. Assume as before that for example $0 \notin B$, the argument is unchanged when $0 \in B$, only the notations have to be adapted. We adopt the same notations as in $(4)$: denote by $C_i$, respectively $C'_j$, $\,  0 \leq i \leq p$, $0\leq j \leq p+1$, the different connected components of $B$, respectively $B^c$, starting from the one closest to the origin. Set $I_{2i} = \{t \geq 0 \mid \gamma(t) \in C_{i-1} \}$ for $1 \leq i \leq p+1$, $I_{2j+1} = \{ t \geq 0 \mid  \gamma(t) \in C'_{j} \cap (2B \backslash B) \}$ for $0 \leq j \leq p+1$. We want to estimate the following quantity
\[ H(f,g) = \int_B \int_{2B \backslash B} \frac{f(x)g(y)}{\Lambda(B(x,\dd (x,y)))} \dd \Lambda(x) \dd \Lambda(y).   \]
As $X_T$ is Ahlfors-David, and applying \eqref{eth}, we have
\begin{equation}  \label{hfgint}
H(f,g) \eqsim \int_B \int_{2B \backslash B} \frac{f(x)g(y)}{\dd (x,y)} \dd \Lambda(x) \dd \Lambda(y) \eqsim \sum_{i=0}^{p} \sum_{j=-1}^{p} \int_{I_{2i}} \int_{I_{2j+1}} \frac{f(\gamma(t))g(\gamma(s))}{|s-t|} \dd t \dd s,
\end{equation}
because $|\gamma'(t)|=1$ for every $t \geq 0$. Set $\widetilde{f} = f \circ \gamma $, $\widetilde{g} = g \circ \gamma$, $\widetilde{f} \in L^{\nu}$, supported on $\bigcup_{i=1}^{p+1} {I_{2i}}$, $\widetilde{g} \in L^{\nu'}$, supported on $\bigcup_{j=0}^{p+1}{I_{2j+1}}$, with $\|\widetilde{f}\|_{\nu} = \|f\|_{\nu}$, $\|\widetilde{g}\|_{\nu'} = \|g\|_{\nu'}$. Let again $k_0\geq 0$ be such that $2^{k_0} \leq \dist(C_0, 0) <  2^{k_0 +1}$ (remember that we have assumed $|z| >r$). Because of the fact that $|\gamma'(t)|=1$ for every $t\geq 0$, observe that we have, by \eqref{taillecomponent}, for every $1 \leq i \leq p+1$, $0 \leq j \leq p$,
\[ |I_{2i}| = \Lambda(C_{i-1}) \leq \Lambda(C_{i-1} \cup C'_{i-1}) \lesssim 2^{k_0}4^i, \quad \quad  |I_{2j+1}| \leq \Lambda(C'_j) \leq \Lambda(C_j \cup C'_j)  \lesssim 2^{k_0}4^j .    \]
For $j=p+1$, remark that as for $0 \leq i \leq p$, $C_i \subset B(0,2^{k_0 +2i+2} )$, we have $B \subset B(0, 2^{k_0+2p+2})$. It follows that $2B\subset B(0, 3 \times 2^{k_0+2p+2})$. As a matter of fact, if $\dd (z,x) < 2r$, then $|x| \leq \dd (x,z) + |z| < 2r + |z| < 3 |z| < 3 \times 2^{k_0+2p+2} $, because we have assumed $|z|>r$. Thus, since $X_T$ is Ahlfors-David, it follows that
\[ |I_{2p+3}| \leq \Lambda(2B \backslash B) \leq \Lambda(B(0, 3 \times 2^{k_0+2p+2})) \lesssim 2^{k_0} 4^p. \]
Finally, it is easy to see that if $j \notin \{i-1,i\}$, we have $\dist(I_{2i},I_{2j+1}) \gtrsim 2^{k_0}|4^i - 4^j|$. Now, set
\[ f_i = \left ( \int_{I_{2i}}{|\widetilde{f}(t)|^{\nu} \dd t } \right )^{1/{\nu}} \in {\ell}^{\nu}(\{ 1,...,p+1 \} ), \quad \mathrm{with} \quad \|(f_i)_i\|_{\ell^{\nu}} = \|\widetilde{f}\|_{\nu}= \|f\|_{\nu},  \]
\[ g_j = \left ( \int_{I_{2j+1}}{|\widetilde{g}(s)|^{\nu'} \dd s } \right )^{1/{\nu'}} \in \ell^{\nu'}(\{0,...,p+1\}),  \quad \mathrm{with} \quad \|(g_j)_j\|_{\ell^{\nu'}} = \|\widetilde{g}\|_{\nu'} =\|g\|_{\nu'}. \]
Split the sum in \eqref{hfgint} for the neighboring $I_k$ and the ones that are far from one \mbox{another:} we have
\begin{align*}
H(f,g) & \lesssim \sum_{i=1}^{p+1} \sum_{j=i-1}^{i} \int_{I_{2i}} \int_{I_{2j+1}} \frac{|\widetilde{f}(t)| | \widetilde{g}(s)|}{|s-t|} \dd t \dd s + \sum_{i=1}^{p+1} \sum_{j\notin \{i-1,i\}} \int_{I_{2i}} \int_{I_{2j+1}} \frac{|\widetilde{f}(t)| | \widetilde{g}(s)|}{|s-t|} \dd t \dd s\\
& = H_1(f,g) + H_2(f,g).
\end{align*}
For $H_1$, apply \eqref{3etoile} on $\R$ and then the Cauchy-Schwarz inequality to get
\begin{align*}
H_1(f,g) & \lesssim \sum_{i=1}^{p+1} \sum_{j \in \{i-1,i\}} {f_i g_j} \lesssim \left ( \sum_{i=1}^{p+1} \sum_{j \in \{i-1,i\}} {f_i^{\nu}} \right )^{1/{\nu}} \left ( \sum_{j=0}^{p+1} \sum_{1 \leq i \leq p+1 \atop i \in \{j+1,j\}} {g_j^{\nu'}} \right )^{1/{\nu'}} \\
& \lesssim \|f\|_{\nu} \|g\|_{\nu'}.
\end{align*}
To estimate $H_2$, write
\begin{align*}
H_2(f,g) & \lesssim \sum_{i=1}^{p+1} \sum_{j\notin \{i-1,i\}} \frac{1}{\dist(I_{2i},I_{2j+1})} f_i |I_{2i}|^{\frac{1}{\nu'}} g_j |I_{2j+1}|^{\frac{1}{\nu}} \\
& \lesssim \sum_{i=1}^{p+1} \sum_{j\notin \{i-1,i\}}  \frac{4^{\frac{i}{\nu'}} 4^{\frac{j}{\nu}}}{|4^i - 4^j|} f_i g_j
\end{align*}
By symmetry, we will be done if we can bound for example the sum for $j>i$. But if $j>i$ observe that $|4^i - 4^j| = 4^i(4^{j-i} -1) \gtrsim 4^j$. Applying the Cauchy-Schwarz inequality, write then
\begin{align*}
\sum_{1\leq i \leq p+1, \, j>i}  \frac{4^{\frac{i}{\nu'}} 4^{\frac{j}{\nu}}}{|4^i - 4^j|} f_i g_j & \lesssim \sum_{1\leq i \leq p+1,\, j>i}  f_i g_j 4^{-\frac{j-i}{\nu'}}  \\
& \lesssim \left ( \sum_{1\leq i \leq p+1, \, j>i}  f_i ^{\nu} 4^{-\frac{j-i}{\nu'}} \right )^{\frac{1}{\nu}} \left ( \sum_{1\leq i \leq p+1,\, j>i}  g_j ^{\nu'}4^{-\frac{j-i}{\nu'}} \right )^{\frac{1}{\nu'}} \\
& \lesssim  \left ( \sum_{1\leq i \leq p+1} f_i ^{\nu} \sum_{j>i} {4^{-\frac{j-i}{\nu'}}} \right )^{\frac{1}{\nu}} \left ( \sum_{2 \leq j \leq p+1} g_j ^{\nu'} 4^{-\frac{j}{\nu'}} \sum_{1 \leq i <j}{4^{\frac{i}{\nu'}}} \right )^{\frac{1}{\nu'}} \\
& \lesssim \left ( \sum_{1\leq i \leq p+1} f_i ^{\nu} \right )^{\frac{1}{\nu}} \left ( \sum_{2 \leq j \leq p+1} g_j ^{\nu'}  \right )^{\frac{1}{\nu'}} \lesssim   \|f\|_{\nu} \|g\|_{\nu'}.
\end{align*}
Finally, one gets $H(f,g) \lesssim  \|f\|_{\nu} \|g\|_{\nu'}$, and thus $X_T$ satisfies the Hardy property $(\mathrm{HP})$.

\end{proof}

\smallskip

Consequently, the space $X_T$ satisfies \eqref{LUOLD} and $(\mathrm{HP})$, but neither $\mathrm{(HB)}$, $\mathrm{(M)}$ nor \eqref{LUAD}. As we have already pointed out before, there is a tangible difference between the definitions of layer decay and annular decay properties. Thus, it is not surprising to find out that these properties are not equivalent. We now have a counterexample for most of the false implications in Theorem \ref{diagram}. It only remains to build a space satisfying \eqref{OLD} but neither \eqref{LUOLD} nor $(\mathrm{HP})$ to complete the proof.

\medskip

\begin{remark}
Let us comment this construction. Observe that the space $X_T$ is a spiral which curls up around the origin in a scale invariant way. Observe as well that we rely on this scale invariance in the above proof. However, one can get another example, basically just by truncating the space $X_T$, of a space satisfying \eqref{LUOLD} but not \eqref{LUAD}. Indeed, let $X'_T = \{\gamma(t) \mid 0\leq t \leq t_8\} \cup [8,+\infty[$ where $[8,+\infty[$ denotes the half-line on the real axis. Then the argument in $(2)$ still holds and $X'_T$ does not satisfy \eqref{LUAD}. On the other hand, it is immediate to see that for any given ball $B$ of $X'_T$, $\card (\overline{B} \backslash B) \leq 6$, and \eqref{LUOLD} follows easily. This shows that the scale invariance is not necessary.
\end{remark}

\bigskip

%
%

\section{Counterexamples and end of the proof of Theorem \ref{diagram}}

\bigskip
\medskip

We now present some variations of the space $X_T$ in order to provide a space where the Hardy property cannot be satisfied. It was originally inspired from the curve of Tessera, but is actually in the end only marginally connected to it. Still in the complex plane, consider the space formed by the union of the segment $[0,1]$ on the real axis, the half circle of center $0$ and radius $1$ in the half-plane $\Im z \geq 0$, and the half-line $]-\infty, -1]$ on the real axis. Parameter this curve so that it is traveled at constant speed. This is someway a truncation of the space $X_T$. Now introduce a small perturbation $\epsilon$ of the half-circle: for $t \in [0,\pi]$, set $\rho(t) = 1 + \epsilon(t)$, where $\epsilon$ is a rapidly oscillating function in the neighborhood of the origin, and set
\[ X_{\epsilon} =   \{ \gamma(t)  \mid    t \geq -1 \}, \quad \mathrm{with} \quad \gamma(t) = \begin{cases} (t+1,0) \quad \mathrm{for} \quad -1 \leq t \leq 0, \\  (\rho(t)\cos t, \rho(t) \sin t) \quad \mathrm{for} \quad 0<t < \pi , \\ (\pi -t,0) \quad \mathrm{for} \quad t \geq \pi .    \end{cases}  \]
Choose an oscillating function $\epsilon$, ensuring that $X_{\epsilon}$ keeps finite arclength: $\epsilon(t) = a(t) \sin( b(t))$ for functions $a,b$ appropriately chosen.

\subsection{Exponential oscillation}

Let
\[ \epsilon_1(t) = e^{-\frac{1}{t^2}} \sin (\pi e^{\frac{1}{t} - \frac{1}{\pi}}), \]
and set $X_1 = X_{\epsilon_1}$. See Figure \ref{FigX1} for a representation of the space $X_1$.

\bigskip

\begin{figure}[htbp]
	\centering
		\includegraphics[scale=0.8]{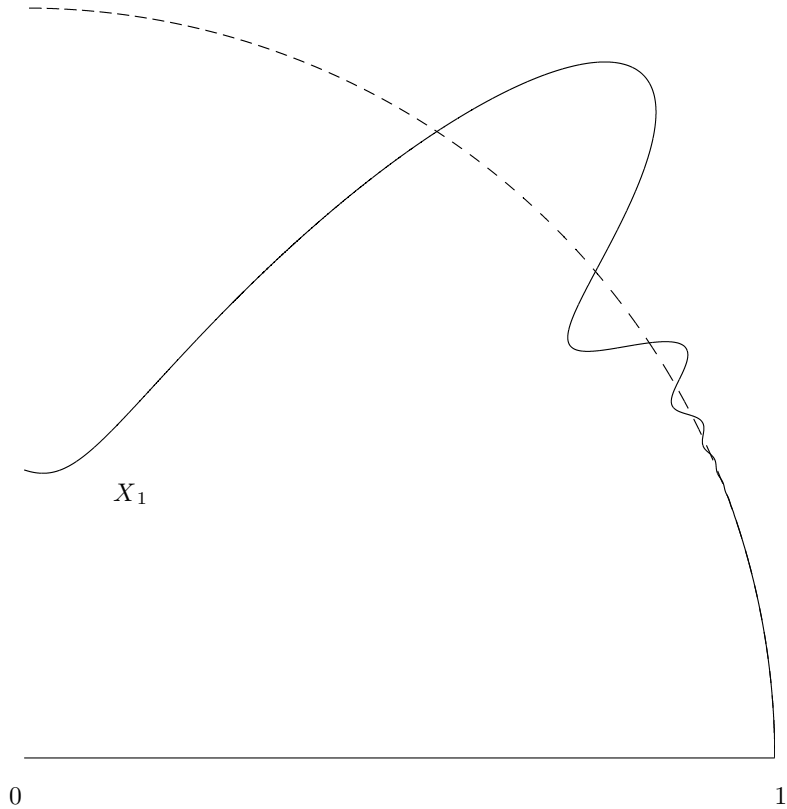}
	\caption{A representation of the space $X_1$.}
	\label{FigX1}
\end{figure}

\bigskip

\noindent Remark that choosing $\frac{1}{t^{\alpha}}$ instead of $\frac{1}{t}$ would not change anything in the following. Observe that with this choice of $\epsilon$, $\gamma$ is $C^1$ and $|\gamma'|$ is uniformly bounded below and above, which obviously makes of $(X_1, \dd , \Lambda)$ an Ahlfors-David space. In the following of this section, $B(z,r)$ will always denote the ball in $\R^2$ of center $z$ and radius $r$. For a ball $B$ in $\R^2$ centered at a point of $X_1 \subset \R^2$, we will denote by $B^{X_1}$ the corresponding ball in the space ${X_1}$. Let $B^0 = B(0,1)$.

\bigskip

\begin{prop} \begin{enumerate}
\item ${X_1}$ does not satisfy the layer decay property \eqref{OLD}.
\item $X_1$ does not satisfy the Hardy property $(\mathrm{HP})$.
\end{enumerate}
\end{prop}

\bigskip

\begin{proof}
We prove that neither of these properties are satisfied for the ball $B^{0,{X_1}}$.

\bigskip

\noindent $(1)$ Let $\varepsilon >0$. Remark that if $\gamma(t) \in \overline{B^{0,{X_1}}} \backslash B^{0,{X_1}} $, then $\epsilon(t)=0$. Denote by $t_k$ these points with $(t_k)_{k\in \N}$ a sequence decreasing to zero, and $t_0 = \pi$. Then we have
\[ B^{0,{X_1}}_{\varepsilon} = \{ \gamma(t)  \mid \exists k \in \N \ \  \dd(\gamma(t),\gamma(t_k)) < \varepsilon  \}. \]
Because of the uniform boundedness of $| \gamma' |$ above and below, observe that we have
\[  \Lambda(B^{0,{X_1}}_{\varepsilon}) \eqsim  | \{ t \in [0,\pi] \mid \exists k \ \  |t_k - t|< \varepsilon \} |, \]
On the other hand, for $k \geq 1$, we have
\[ \pi e^{-\frac{1}{\pi}} e^{\frac{1}{t_k}} = k \pi \Rightarrow t_k = \frac{1}{\ln k + 1/{\pi}} \Rightarrow t_k - t_{k+1}  \lesssim \frac{1}{k}.  \]
Thus, if $k \geq \frac{1}{\varepsilon}$, and $t \leq t_k \lesssim \frac{1}{-\ln \varepsilon}$, then $t_k - t_{k+1} \lesssim  \varepsilon$. It implies that $\gamma(t)$ stays inside $B^{0,{X_1}}_{\varepsilon}$ for all the $t\leq \frac{C}{-\ln \varepsilon}$ for some constant $C < +\infty$. We obtain
\[ \Lambda(B^{0,{X_1}}_{\varepsilon}) \gtrsim \left | \left \{ t \in [0,\pi] \mid \exists k \geq \frac{1}{\varepsilon} \ \  |t_k-t|< \varepsilon \right  \} \right | \geq  \left | \left [0, \frac{C}{-\ln \varepsilon} \right ] \right | \gtrsim  \frac{1}{-\ln \varepsilon}.  \]
It follows that there cannot be any upper bound of the form $\varepsilon ^{\eta}$ for $\Lambda(B^{0,{X_1}}_{\varepsilon})$, and \eqref{OLD} cannot be satisfied. Geometrically, this shows that the concentration of points in the layers of the unit ball is too important for $X_1$ to satisfy the layer decay property.\\

\bigskip

\noindent $(2)$ Let $1<\nu<+\infty$. Let $f \in L^{\nu}(B^{0,{X_1}})$, $f$ supported on $B^{0,{X_1}}$, $g\in L^{\nu'}(2B^{0,{X_1}}\backslash B^{0,{X_1}})$, $g$ supported on $2B^{0,{X_1}} \backslash B^{0,{X_1}}$. Denote by $I_{2k}, k \geq0$ the connected sets of $t$ for which $\gamma(t) \in B^{0,{X_1}}$, $I_{2k} = ]t_{2k},t_{2k+1}[$, and $I_{2p+1}, p\geq0$ the connected sets of $t$ for which $\gamma(t) \in 2B^{0,{X_1}} \backslash B^{0,{X_1}}$, $I_{2p+1} = [t_{2p+1}, t_{2p+2}]$. By the Ahlfors-David property of ${X_1}$, we have
\[ |H(f,g)| = \left | \int_{B^{0,{X_1}}} \int_{2B^{0,{X_1}} \backslash B^{0,{X_1}} } \frac{f({x})g(y)}{\lambda({x},y)} \dd \Lambda(y) \dd \Lambda({x}) \right | \eqsim \left | \sum_{k,p \geq0} \int_{t \in I_{2k}} \int_{s \in I_{2p+1} } \frac{\widetilde{f}(t)\widetilde{g}(s)}{|t-s|} \dd s \dd t \right |, \]
where $\widetilde{f} = f \circ \gamma $, $\widetilde{g} = g \circ \gamma$, $\widetilde{f} \in L^{\nu}$, supported on $\cup_{k\geq0} {I_{2k}}$, $\widetilde{g} \in L^{\nu'}$, supported on $\cup_{p\geq0}{I_{2p+1}}$, and $\|\widetilde{f}\|_{\nu} \eqsim \|f\|_{\nu}$, $\|\widetilde{g}\|_{\nu'} \eqsim \|g\|_{\nu'}$, because $|\gamma'|\eqsim 1 $. Now, assume that $\widetilde{f},\widetilde{g}$ are constant and positive on each $I_{2k}, I_{2p+1}$. Set
\[ f_k = \left ( \int_{I_{2k}}{|\widetilde{f}(t)|^{\nu} \dd t } \right )^{1/{\nu}} \in {\ell}^{\nu}(\N), \quad \mathrm{with} \quad \|(f_k)_k\|_{\ell^{\nu}} = \|\widetilde{f}\|_{\nu},  \]
\[ g_p = \left ( \int_{I_{2p+1}}{|\widetilde{g}(s)|^{\nu'} \dd s } \right )^{1/{\nu'}} \in \ell^{\nu'}(\N),  \quad \mathrm{with} \quad \|(g_p)_p\|_{\ell^{\nu'}} = \|\widetilde{g}\|_{\nu'}. \]
Remark that if $p,k \geq 1$ with $p\notin \{ k,k-1\}$, and $t \in I_{2k}$, $s \in I_{2p+1}$, then
\[ |s-t| \lesssim  \left |\frac{1}{\ln p} - \frac{1}{\ln k} \right | \lesssim  \frac{|\ln (p/k)|}{\ln p \ln k}.    \]
We thus have
\begin{align*}
|H(f,g)|&  \gtrsim \sum_{k,p \geq1 \atop p \notin \{k,k-1\}} \frac{\ln p \ln k}{|\ln(p/k)|} \int_{t \in I_{2k}} |\widetilde{f}(t)| \dd t \int_{s \in I_{2p+1} } |\widetilde{g}(s)| \dd s \\
& =  \sum_{k,p \geq1 \atop p \notin \{k,k-1\}} f_k g_p \frac{\ln p \ln k}{|\ln(p/k)|}  |I_{2k}|^{\frac{1}{\nu'}} |I_{2p+1}|^{\frac{1}{\nu}}.
\end{align*}
But since
\[ |I_l| = |t_{l+1} - t_l | \underset {l \rightarrow +\infty} \sim \frac{1}{l (\ln l)^2},  \]
there exists $N \in \N^{\ast}$ such that if $l \geq N$, then $|I_{l}| \geq  \frac{1}{2 l (\ln l)^2}$. It follows that
\[ |H(f,g)| \gtrsim  \sum_{k,p \geq N \atop p \notin \{k,k-1\}} f_k g_p \frac{(\ln p)^{1- \frac{2}{\nu}} (\ln k)^{1-\frac{2}{\nu'}}}{k^{\frac{1}{\nu'}} p^{\frac{1}{\nu}}|\ln(p/k)|}. \]
It is easy to see that this is an unbounded operator. Fix $0<\eta<1/2$, and set for example for $k,p \geq N$
\[ f_k = \frac{1}{k^{\frac{1}{\nu}} (\ln k)^{\frac{1}{\nu}+ \eta}}, \quad g_p = \frac{1}{p^{\frac{1}{\nu'}} (\ln p)^{\frac{1}{\nu'}+ \eta}}. \]
Then
\begin{align*}
\sum_{p \geq k+1} g_p \frac{1}{p^{\frac{1}{\nu}}} \frac{(\ln p)^{1-\frac{2}{\nu}}}{|\ln(p/k)|} & = \sum_{p \geq k+1} \frac{1}{p(\ln p)^{\frac{1}{\nu} +\eta}\ln (p/k)} \\
& \geq \int_{k+1}^{+\infty} \frac{\dd t}{t (\ln t)^{\frac{1}{\nu} + \eta} \ln(t/k)} = \int_{1+\frac{1}{k}}^{+\infty}\frac{\dd u}{u \ln u \ln(ku)^{\frac{1}{\nu}+\eta}} \\
& \geq \frac{1}{\ln(3k)^{\frac{1}{\nu}+ \eta}} \int_{1+ \frac{1}{k}}^3 \frac{\dd u}{u \ln u}  \\
& \gtrsim \frac{|\ln (\ln(1+1/k))|}{(\ln k)^{\frac{1}{\nu} +\eta}} \underset{k \rightarrow +\infty} \sim (\ln k)^{\frac{1}{\nu'} - \eta}.
\end{align*}
It follows that
\begin{align*}
|H(f,g)| \gtrsim \sum_{k \geq N} f_k \frac{(\ln k)^{1- \frac{2}{\nu'}}}{k^{\frac{1}{\nu'}}}  (\ln k)^{\frac{1}{\nu'} - \eta} = \sum_{k\geq N} \frac{1}{k(\ln k)^{2 \eta}} = +\infty.
\end{align*}
Thus \eqref{3etoile} cannot be satisfied for any $1<\nu<+\infty$. Once again, geometrically there is too much mass that concentrates in the inner and outer layers of the unit ball of $X_1$ for the Hardy property to be satisfied.
\end{proof}

\bigskip

\begin{remark}
Observe that although the layer decay property is not satisfied here, we still have $\Lambda(B^{0,X_1}_{\varepsilon}) \underset{\varepsilon \rightarrow 0} \rightarrow 0$. Indeed, $\bigcap_{n \geq 1} B^{0,X_1}_{1/n} = \overline{B^{0,X_1}} \backslash B^{0,X_1} $, but this set is of measure $0$ as it is countable. Thus, we have
\[ \Lambda(\overline{B^{0,X_1}}\backslash B^{0,X_1} ) = 0 = \lim_{n \rightarrow +\infty} \Lambda(B^{0,X_1}_{1/n}).  \]
Observe furthermore that it is always the case in this kind of example with a continuous function $\epsilon$, because $\overline{B^{0,X_{\epsilon}}} \backslash B^{0,X_{\epsilon}}$ must necessarily be a countable set, hence of measure zero. Indeed, $\overline{B^{0,X_{\epsilon}}} \backslash B^{0,X_{\epsilon}}$ is the set of points $\gamma(t_0)$ for which $\epsilon(t_0)=0$ and for every $\eta>0$, there exists $0<\varepsilon<\eta$ such that $\epsilon(t_0 \pm \varepsilon) <0$. By continuity of $\epsilon$, we deduce from this that for every point $\gamma(t_0) \in \overline{B^{0,X_{\epsilon}}} \backslash B^{0,X_{\epsilon}}$, there exists $q_0 \in \Q$ with $\epsilon(q_0)<0$ and $|t_0-q_0|$ as small as one wants. We can thus construct an injection from $\overline{B^{0,X_{\epsilon}}} \backslash B^{0,X_{\epsilon}}$ to $\Q$ and it follows that $\overline{B^{0,X_{\epsilon}}} \backslash B^{0,X_{\epsilon}}$ is countable, hence of measure zero.
\end{remark}

\bigskip

\subsection{Polynomial oscillation}

This time, let $b(t) = \frac{\pi^2}{t}$ : choose
\[ \epsilon_2(t) =  A_0 \left (\frac{t}{\pi} \right )^3 \sin \left ( \frac{\pi^2}{t} \right ), \]
with $A_0$ a sufficiently small constant to be specified later, and construct a space $X_2$ as before. Again, $({X_2},\dd , \Lambda)$ is an Ahlfors-David space (and thus a space of homogeneous type).

\bigskip

\begin{prop}\begin{enumerate}
\item ${X_2}$ does not satisfy the Hardy property $(\mathrm{HP})$.
\item $X_2$ does satisfy the layer decay inequality \eqref{OLD}, but not the relative layer decay inequality \eqref{LUOLD}.
\end{enumerate}
\end{prop}

\begin{proof}
\noindent $(1)$ We prove again that \eqref{3etoile} is not satisfied for the unit ball $B^{0,X_2}$. Let us use the same notations as before. For functions $f \in L^2(B^{0,{X_2}})$, $f$ supported inside $B^{0,{X_2}}$, and $g \in L^2(2B^{0,{X_2}} \backslash B^{0,{X_2}})$, $g$ supported inside $2B^{0,{X_2}} \backslash B^{0,{X_2}}$, we have
\[ |H(f,g)| \eqsim \left | \sum_{k,p \geq0} \int_{t \in I_{2k}} \int_{s \in I_{2p+1} } \frac{\widetilde{f}(t)\widetilde{g}(s)}{|t-s|} \dd s \dd t \right |. \]
This time, we have for $k \geq 1$, $t_k = \frac{\pi}{k} $, and thus for $l \geq 1$, we have $|I_l | = \frac{\pi}{l(l+1)} \gtrsim \frac{1}{l^2}$. Besides, if $p,k \geq 1$ with $p\notin \{k,k-1\}$, and $t \in I_{2k}$, $s \in I_{2p+1}$, then
\[ |s-t|^{-1} \gtrsim  \left |\frac{1}{p} - \frac{1}{k} \right |^{-1} \gtrsim  \frac{pk}{|p-k|}.    \]
Assume again that $\widetilde{f},\widetilde{g}$ are constant and positive on each $I_{2k}, I_{2p+1}$, and set
\[ f_k = \left ( \int_{I_{2k}}{|\widetilde{f}(t)|^{2} \dd t } \right )^{1/{2}} \in {\ell}^{2}(\N), \quad \mathrm{with} \quad \|(f_k)_k\|_{\ell^{2}} = \|\widetilde{f}\|_{2},  \]
\[ g_p = \left ( \int_{I_{2p+1}}{|\widetilde{g}(s)|^{2} \dd s } \right )^{1/{2}} \in \ell^{2}(\N),  \quad \mathrm{with} \quad \|(g_p)_p\|_{\ell^{2}} = \|\widetilde{g}\|_{2}. \]
Then, we have
\[ |H(f,g)| \gtrsim \sum_{k,p \geq 1 \atop p \notin \{k,k-1\}} \frac{f_k g_p}{|p-k|}.  \]
But it is well known that this operator is unbounded on $\ell^2$. It follows that \eqref{3etoile} cannot be satisfied for $\nu =2$. Thus, $X_2$ does not satify $(\mathrm{HP})$, nor \eqref{LUOLD} because of Proposition \ref{LUOLDP}. 

\bigskip

\noindent $(2)$ Let $\varepsilon >0$. We are going to prove \eqref{OLD} for all the balls $B^{X_2}(z,r)$ centered at a point $z \in X_2$ of radius $r>0$. We classify these balls in three categories, each of which will be taken care of differently: first there are the balls $B(z,r)$ of radius $r \geq 1/2$, then there are the balls $B(z,r)$ of radius $0<r<1/2$ tangential to the ball $B^0$ at the point of affix $1$, and finally the balls $B(z,r)$ of radius $0<r<1/2$ non tangential to the ball $B^0$ at the point of affix $1$. We begin by taking care of the first category. We first show that ${X_2}$ satisfies \eqref{OLD} for the unit ball $B^{0,X_2}$, for some exponent $\eta < 1$. Indeed, remark that if $k \geq  \varepsilon^{-1/2}$, then $|I_k| = |t_k - t_{k+1}| \lesssim \frac{1}{k^2} \lesssim \varepsilon$, and it implies that $\gamma(t)$ stays inside $B^{0,{X_2}}_{\varepsilon}$ for all the $t \leq C \varepsilon^{1/2}$ for some uniform constant $0<C<+\infty$. Besides, observe that we have
\[ \mathrm{Card} \left \{ k \in \N \mid k \lesssim \varepsilon^{-1/2} \right \} \lesssim \varepsilon^{-1/2}, \]
and that for each one of the corresponding $t_k$, there is a contribution of at most $\varepsilon$ to $\Lambda(B^{0,{X_2}}_{\varepsilon})$. Thus, we have
\begin{align*}
\Lambda(B^{0,{X_2}}_{\varepsilon}) & \lesssim |[0, C \varepsilon^{1/2} ]| + \varepsilon \times  \mathrm{Card} \left \{ k \in \N \mid k \lesssim \varepsilon^{-1/2} \right \} \lesssim \varepsilon^{1/2} + \varepsilon \times \varepsilon^{-1/2} \lesssim \varepsilon^{1/2}.
\end{align*}
Since $X_2$ is Ahlfors-David, we have $ \Lambda(B^{0,{X_2}}) \eqsim 1$, and \eqref{OLD} follows.

Now, observe that this extends to all the balls $B=B(z,r)$ of radius $r \geq 1/2$. As a matter of fact, remark that we necessarily have $ \Lambda(B^{X_2}_{\varepsilon}) \leq \Lambda(B^{0,X_2}_{\varepsilon}) \lesssim \varepsilon^{1/2}$. Indeed, there are at most two elements in $\overline{B^{X_2}} \backslash B^{X_2}$ outside of $\{ \gamma(t) \mid 0 \leq t \leq \pi \}$. And it is also easy to see that $\overline{B^{X_2}} \backslash B^{X_2} \cap \{ \gamma(t) \mid 0 \leq t \leq \pi \}$ can be injected inside $\overline{B^{0,X_2}} \backslash B^{0,X_2}$. Thus the preceding argument still applies. It follows that
\[ \Lambda(B^{X_2}_{\varepsilon}) \lesssim \left ( \frac{\varepsilon}{r} \right)^{1/2} r^{1/2} \lesssim \left ( \frac{\varepsilon}{r} \right)^{1/2} r , \]
because $r\geq 1/2$. But since $X_2$ is Ahlfors-David, we have $ \Lambda(B^{X_2}) \eqsim r$ and \eqref{OLD} follows.

Now, we consider the balls $B=B(z,r)$ with $z=1-r$ and $0<r<1/2$, tangential to $B^0$ at the point of affix $1$. Let $\mathcal{C}$ denote in $\R^2$ the circle of center $z$ and radius $r$. Switching to polar coordinates, for $t>0$ sufficiently small ($t \leq t_{\mathrm{max}} = \arctan{\frac{r}{1-r}}$), denote by $M(t) =(t,v(t)) $ the point of the circle $\mathcal{C}$ farthest from the origin, and let $u(t) = 1- v(t)$ (see Figure \ref{FigproofX2}). Then $v(t)$ satisfies the following equation
\[ v(t)^2  -2 v(t) (1-r) \cos t + (1-r)^2 = r^2,\]
so that we have
\[ u(t)^2 +2u(t) [(1-r)\cos t -1] +2(1-r)[1-\cos t] = 0, \]
hence
\[ u(t) = 1 - (1-r) \cos t - [(1-r)^2 \cos^2 t +1 -2(1-r)]^{1/2}. \]

\bigskip

\begin{figure}[htbp]
	\centering
		\includegraphics[scale=0.65]{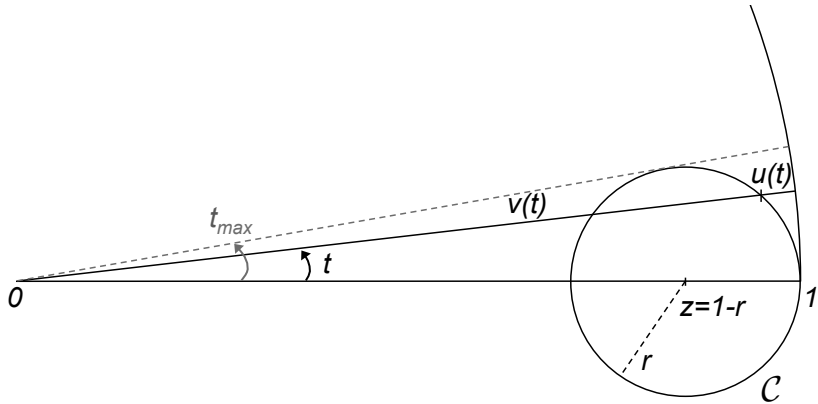}
	\caption{Existence of a separation between $\mathcal{C}$ and the unit circle.}
	\label{FigproofX2}
\end{figure}

\bigskip
\bigskip

\noindent There exists $t_0 >0$ such that for every $0<t<t_0$, $\cos t \leq 1 - \frac{t^2}{4}$ and $t^4 \leq t^2$. Then, for every $0< t <t_0$ and $0<r<1/2$, we have
\begin{align*}
u(t) & \geq 1 -(1-r) (1- \frac{t^2}{4}) - \left [(1-r)^2 (1- \frac{t^2}{4})^2 +2r -1 \right ]^{1/2} \\
& \geq r + \frac{1-r}{4} t^2 - \left [(1-r)^2 (1- \frac{t^2}{2} + \frac{t^4}{16}) +2r -1 \right ]^{1/2}\\
& \geq r + \frac{1-r}{4} t^2 - r \left  [1- \frac{(1-r)^2}{2r^2} t^2 + \frac{(1-r)^2}{16r^2} t^4 \right ]^{1/2} \\
& \geq  r + \frac{1-r}{4} t^2 - r \left  [1- \frac{7(1-r)^2}{16r^2} t^2 \right  ]^{1/2} \geq r + \frac{1-r}{4} t^2 -r \left [ 1 - \frac{1}{2} \frac{7(1-r)^2}{16r^2} t^2  \right ] \\
& \geq \frac{7 - 6r -r^2}{32r} t^2  \geq \frac{3}{16}t^2.
\end{align*}
Furthermore, since $u(t)$ clearly is an increasing function, if $t_0 \leq t \leq t_{\mathrm{max}}$, we have $u(t) \geq \frac{3}{16} t_0^2$. Now, as $|\epsilon(t)| \leq A_0 (t/{\pi})^3$, it is clear that if $A_0 < \frac{3}{16}t_0^2$, then $\gamma(t)$ stays outside of $B$ for every $0<t \leq \pi$. Thus, $B^{X_2}$ reduces to the open segment $]1-2r, 1[$ which is connected. It implies that $\mathrm{Card} ( \overline{B^{X_2}} \backslash B^{X_2} )\leq 2$, and consequently, we have
\[ \Lambda (B^{X_2}_{\varepsilon}) \lesssim 2 \varepsilon \lesssim \left ( \frac{\varepsilon}{r}\right ) \Lambda (B^{X_2}), \]
as, once again, by the Ahlfors-David property of $X_2$ we have $\Lambda (B^{X_2}) \simeq r$.

It remains only to consider the balls $B=B(z,r)$ of radius $0<r<1/2$ non tangential to the ball $B^0$ at the point of affix $1$. But it is immediate to see that the number of connected components of such a ball $B^{X_2}$ is at most $2$: $B^{X_2}$ is a connected set in most cases, but there can be two connected components if $z= \gamma(\tau)$ with $0<\tau \leq \pi$ is small and $r$ is small enough (then $1 \notin B^{X_2}$ but $\gamma(t) \in B^{X_2}$ for some $-1<t<0$). Thus, as for the balls of the previous category, we have $\mathrm{Card} ( \overline{B^{X_2}} \backslash B^{X_2}) \leq 4$, and
\[ \Lambda (B^{X_2}_{\varepsilon}) \lesssim 4 \varepsilon \lesssim \left ( \frac{\varepsilon}{r}\right ) \Lambda (B^{X_2}). \]
Putting all this together, we have proven that ${X_2}$ satisfies \eqref{OLD} for $\eta =1/2$.
\end{proof}

\bigskip
\bigskip

\noindent Let us make a few geometric comments about this result. Observe that $X_2$ satisfying the layer decay property \eqref{OLD} means that the measures of the outer and inner layers of the balls in $X_2$, and more particularly the unit ball, do not get too big. The amplitude of the polynomial oscillation around the unit ball does not decrease too rapidly, so that the mass does not concentrate too much in those layers. On the other hand, this is not the case locally, and particularly in the neighborhood of the point of affix $1$, where the mass does concentrate heavily in the layers. This is why the relative layer decay property \eqref{LUOLD} fails.\\

\bigskip

\noindent The space $X_2$ is thus a counterexample to the implications $\eqref{OLD} \Rightarrow \eqref{LUOLD}$ and $\eqref{OLD} \Rightarrow (\mathrm{HP})$. The proof of Theorem \ref{diagram} is now complete.

\bigskip

\begin{remarks}\begin{itemize}
\item Choosing $b(t) = \frac{\pi^{1+\alpha}}{t^{\alpha}} $ with $\alpha>0$, and $a(t)$ accordingly, would give a space very similar to $X_2$, satisfying the same properties. Actually, the choice of $a$ does not really matter as long as it ensures that $\gamma$ keeps finite arclength.
\item One could pick similar examples for other types of decreasing functions $b$. This range of examples shows that the Hardy property, as well as the relative and non relative layer decay properties, are very unstable, as it suffices to apply a slight perturbation to the initial space, where they were satisfied, to lose them.
\end{itemize}
\end{remarks}

\bigskip
\bigskip
\bigskip
\bigskip
\bigskip

\bibliographystyle{short}	
\bibliography{myrefs}

\begin{thebibliography}{LNY}

\bibitem[A]{Ahl}
L.~Ahlfors.
\newblock {Zur Theorie der Uberlagerungsfl{\"a}schen}.
\newblock {\em Acta Math.}, 65:157--194, 1935.

\bibitem[AH]{AHyt}
P.~Auscher and T.~Hyt{\"o}nen.
\newblock Orthonormal bases of regular wavelets in spaces of homogeneous type.
\newblock {\em Preprint, arXiv:1110.5766v2}, 2011.

\bibitem[AR]{AR}
P.~Auscher and E.~Routin.
\newblock {Local Tb theorems and Hardy inequalities}.
\newblock {\em J. Geom. Anal.}, 23(1):303--374, 2013.

\bibitem[B1]{Buck2}
S.~M. Buckley.
\newblock {Inequalities of John-Nirenberg type in doubling spaces.}
\newblock {\em J. Anal. Math.}, 79:215--240, 1999.

\bibitem[B2]{Buck}
S.~M. Buckley.
\newblock {Is the maximal function of a Lipschitz function continuous?}
\newblock {\em Ann. Acad. Sci. Fenn. Math.}, 24:519--528, 1999.

\bibitem[BBI]{Burago}
D.~Burago, Y.~Burago, and S.~Ivanov.
\newblock {\em A Course in Metric Geometry}, volume~33 of {\em Graduate studies
  in mathematics}.
\newblock American Mathematical Society, 2001.

\bibitem[C]{Christ}
M.~Christ.
\newblock {A T(b) theorem with remarks on analytic capacity and the Cauchy
  integral.}
\newblock {\em Colloq. Math.}, 60/61:601--628, 1990.

\bibitem[CM]{CM}
T.~H. Colding and W.~P. {Minicozzi II}.
\newblock {Liouville theorems for harmonic sections and applications}.
\newblock {\em Comm. Pure Appl. Math.}, 51:113--138, 1998.

\bibitem[CW]{CoifWeiss}
R.~Coifman and G.~Weiss.
\newblock {\em Analyse harmonique non-commutative sur certains espaces
  homog{\`e}nes}, volume 242 of {\em Lecture Notes in Math.}
\newblock Springer-Verlag, Berlin, 1971.

\bibitem[D]{David-LN}
G.~David.
\newblock {\em Wavelets and singular integrals on curves and surfaces}, volume
  1465 of {\em Lecture Notes in Math.}
\newblock Springer-Verlag, Berlin, 1991.

\bibitem[DJS]{DJS}
G.~David, J.~Journ{\'e}, and S.~Semmes.
\newblock {Op{\'e}rateurs de Calder\'on-Zygmund, fonctions para-accr{\'e}tives
  et interpolation.}
\newblock {\em Rev. Mat. Iberoamericana}, 1:1--56, 1985.

\bibitem[H]{Hofmann}
S.~Hofmann.
\newblock {A proof of the local Tb theorem for standard Calder\'on-Zygmund
  operators.}
\newblock Unpublished, arXiv:0705.0840v1, 2007.

\bibitem[LNY]{LNY}
H.~Lin, E.~Nakai, and D.~Yang.
\newblock Boundedness of lusin-area and $g^{\ast}_{\lambda}$ functions on
  localized bmo spaces over doubling metric measure spaces.
\newblock {\em Bull. Sci. Math. (to appear), arXiv:0903.4587v2}, 2010.

\bibitem[MS]{MaSe}
R.A. Mac\'ias and C.~Segovia.
\newblock Lipschitz functions on spaces of homogeneous type.
\newblock {\em Adv. in Math.}, 33(3):257--270, 1979.

\bibitem[PS]{PalStem}
M.~Paluszy\'nski and K.~Stempak.
\newblock On quasi-metric and metric spaces.
\newblock {\em Proc. Amer. Math. Soc.}, 137(12):4307--4312, 2009.

\bibitem[S]{Stein93}
Elias~M. Stein.
\newblock {\em Harmonic analysis: real-variable methods, orthogonality, and
  oscillatory integrals}.
\newblock Number~43 in Princeton Mathematical Series. Princeton University
  Press, Princeton, NJ, 1993.

\bibitem[T]{Tess}
R.~Tessera.
\newblock {Volume of spheres in doubling metric measured spaces and in groups
  of polynomial growth}.
\newblock {\em Bull. Soc. Math. France}, 135:47--64, 2007.

\end{thebibliography}

\end{document}